\documentclass[10pt]{amsart}
\usepackage{amsmath}
\usepackage{times}

\usepackage{amsthm}
\usepackage{amssymb}

\usepackage{amscd}
\usepackage[all]{xy}
\usepackage{color}
\usepackage{verbatim}
\usepackage{stmaryrd}

\usepackage[plainpages=false,colorlinks,hyperindex,pdfpagemode=None,bookmarksopen,linkcolor=red,citecolor=blue,urlcolor=blue]{hyperref}
\usepackage{pdflscape}

\usepackage[OT2,T1]{fontenc}
\DeclareSymbolFont{cyrletters}{OT2}{wncyr}{m}{n}
\DeclareFontFamily{OT1}{rsfs}{}

\usepackage{verbatim}
\usepackage{graphicx}
\usepackage{nomencl}
\usepackage{xcolor}

\DeclareFontEncoding{OT2}{}{} 

\newcommand{\Z}{\mathbf{Z}}

\newcommand{\C}{\mathbf{C}}

\newcommand{\SL}{\mathrm{SL}}

\newcommand{\R}{\mathbf{R}}
\newcommand{\g}{\mathfrak{g}}
\newcommand{\kg}{\mathfrak{k}}
\newcommand{\ag}{\mathfrak{a}}
\newcommand{\pg}{\mathfrak{p}}
\newcommand{\nn}{\mathfrak{n}}

  \newcommand{\D}{\mathbf{D}}

\DeclareFontShape{OT1}{rsfs}{n}{it}{<-> rsfs10}{}
\DeclareMathAlphabet{\mathscr}{OT1}{rsfs}{n}{it}

\newcommand{\vol}{\mathrm{vol}}

\newtheorem{theorem}{Theorem}
\newtheorem{conjecture}[theorem]{Conjecture}

\newtheorem{prop}[theorem]{Proposition}
\newtheorem{lem}[theorem]{Lemma}
\theoremstyle{definition}

\newtheorem{problem}{Problem}
\newtheorem{definition}[theorem]{Definition}

\newtheorem*{unremark}{Remark}

\definecolor{darkgreen}{RGB}{0,150,0}

\newcommand{\N}{\mathbf{N}}

\newcommand{\SO}{\mathrm{SO}}

\newcommand{\waveker}{{\rho_{\Gamma\backslash G}(k_t)}}

\numberwithin{equation}{section}

\title{Eigenfunctions and Random Waves in the Benjamini-Schramm limit}

\author{Miklos Abert}
\address {Renyi Institute of Mathematics \\
13-15 Realtanoda utca, 
1053 Budapest, Hungary\\}
\email{abert.miklos@renyi.mta.hu}

\author{Nicolas Bergeron}
\address{ENS / PSL University, D\'epartement de Math\'ematiques et Applications, F-75005, Paris, France}
\email{nicolas.bergeron@ens.fr}
\urladdr{https://sites.google.com/view/nicolasbergeron/accueil}

\author{Etienne Le Masson}
\address{CY Cergy Paris Universit\'e, AGM, 2 av. Adolphe Chauvin, 95302 Cergy-Pontoise Cedex, France}
\email{etienne.le-masson@cyu.fr}

\begin{document} 

\maketitle

\begin{abstract}
We investigate the asymptotic behavior of eigenfunctions of the Laplacian on Riemannian
manifolds. We show that Benjamini-Schramm convergence provides a
unified language for the level and eigenvalue aspects of the theory.
As a result, we present a mathematically precise formulation of
Berry's random wave conjecture for a compact negatively curved manifold and
formulate a Berry-type conjecture for sequences of locally symmetric
spaces. We prove some weak versions of these conjectures. Using
ergodic theory, we also analyze the connections of these conjectures
to Quantum Unique Ergodicity.
\end{abstract}

\section{Introduction}

\subsection*{Berry's random wave conjecture and Benjamini-Schramm convergence}

Since the seminal work of Berry \cite{Berry}, eigenfunctions of the Laplacian in chaotic settings are believed to behave like Gaussian random waves, in the large eigenvalue limit. 

This has been mostly addressed numerically, however in \cite[Section 5]{Hejhal} the conjectural randomness is quantified by studying the amplitude distribution of high energy eigenfunctions on a closed Riemanniann manifold $M$. 
Viewing $(M, d\mathrm{vol}_M)$ as a probability space each eigenfunction indeed yields an $\R$-valued random variable, and  
the Hejhal--Rackner random wave conjecture asserts that this random variable converges in a suitable sense to the normal distribution with mean $0$ and variance $1/\sqrt{\mathrm{vol}(M)}$.
Here `suitable sense' can have several different meanings, the weakest is the convergence in distribution, or weak convergence, but one can also ask for convergence of moments. These give respectively the Gaussian distribution conjecture and the Gaussian moments conjecture, see \cite[\S 1.2.2 \& 1.2.3]{Humphries}. These conjectures focus only on the values of the eigenfunctions and essentially forget their `shape'.  

We propose below a new formulation of the random wave conjecture that takes into account both the shape of the eigenfunction and its distribution of values. This makes heavy use of a notion of Benjamini--Schramm (BS) sampling in the general Riemannian setting. In Section \ref{S:BS} we first recall the definition \cite{AbertBiringer} of BS convergence for sequences $(M_n )$ of compact connected complete Riemannian $d$-manifolds. A limit point for this convergence is a probability measure on the space of pointed complete Riemannian $d$-manifolds that encodes what the geometry of $M_n$, for large $n$, looks like near randomly chosen points. We then extend the notion of BS convergence to sequences $(M_n , \phi_n )$ of manifolds equipped with a smooth function. A limit point is then a probability measure on the space of pointed complete Riemannian $d$-manifolds equipped with a smooth function.

Now, for a $d$-dimensional Riemannian manifold $M = (M,g)$, let $M_r = (M, g_r)$ denote
the rescaling of $M$ by $r$, which means that we change only the metric by multiplying every distance by $r$. It can be shown that, as $r\rightarrow \infty$, the manifolds $M_r$ BS-converge to $\R^{d}$. Moreover, if $\phi:M\rightarrow \R$ is an eigenfunction of the
Laplacian on $M$ with eigenvalue $\lambda$, then the very same function 
is an eigenfunction of the Laplacian on $M_r$
with eigenvalue $\lambda^{\prime}=\lambda / r^{2}$. We can then formulate the following conjecture.

\begin{conjecture}[Berry's conjecture in BS form] \label{BerryConj}
Let $M$ be a compact, negatively curved
manifold. Let $(\phi_{n} )$ be an orthonormal basis of $L^2 (M)$ that consists of eigenvectors for the Laplace operator, with eigenvalues $\alpha_n^2$. Then $(M_{\alpha_n}, \phi_{n})$ BS converges to the isotropic monochromatic Gaussian random Euclidean wave with eigenvalue $1$.
\end{conjecture}
Here we think of the isotropic monochromatic Gaussian random Euclidean wave with eigenvalue $1$ as a probability measure on the space of smooth functions on $\R^d$ --- see  Section \ref{S:berry} for a detailed discussion of this conjecture.

\begin{unremark} 
For chaotic billiards, Ingremeau \cite{Ingremeau} recently and independently
formulated a conjecture of the same nature. Using results of Bourgain,
Buckley and Wigman \cite{Bou14,BW15} he also proved that certain deterministic
families of eigenfunctions on the $2$-torus satisfy the conclusion of
Berry's conjecture. Note that in this case, the curvature is 0 and no
chaotic dynamics are present.
\end{unremark}

\medskip

Conjecture \ref{BerryConj} connects the eigenvalue aspect of the theory (where we have a fixed manifold and the eigenvalues tend to infinity) to the so-called level aspect, where we deal with a sequence of manifolds and the eigenvalue is (close to) a constant. This new connection makes it natural to formulate Berry's conjecture in the level aspect, or, more generally, for Benjamini-Schramm convergent sequences of manifolds in general. 

In this paper we are mainly concerned with locally symmetric spaces $\Gamma \backslash X$ --- where $X=G/K$ is a Riemannian symmetric space associated with a connected center-free semi simple Lie group without compact factors, $K \subset G$ is a maximal compact subgroup and $\Gamma \subset G$ is uniform lattice in $G$. A particular important case is that of sequences of closed hyperbolic $d$-manifolds --- in this case $G= \mathrm{SO}_0 (d,1)$ and $K= \SO(d)$. We shall denote by $x_0 \in X$ the base point $eK \in G/K$.

To formulate the analogue of Berry's conjecture in this context we define isotropic monochromatic Gaussian random waves on $X$ in Section \ref{S2}. Now consider a sequence $M_n = \Gamma_n \backslash X$ of closed $X$-manifolds. Let $d\mathrm{vol}_{M_n}$ denote the normalized
volume on $M_n$, so that $d\mathrm{vol}_{M_n}$ is a probability measure. Let $(\phi_j ^{(n)} )$ be an orthogonal basis of $L^2 (M_n)$ such that for every $j$ we have: 
$$\int_{M_n } |\phi_j^{(n)} |^2 d\mathrm{vol}_{M_n} = 1 \quad \mbox{ and } \quad \Delta \phi_j^{(n)} = \lambda_j^{(n)} \phi_j^{(n)} \ \left(\lambda_j^{(n)}  \geq 0 \right).$$
We furthermore assume that the eigenvalues are ordered in non-decreasing order. 

In this context we replace the hypothesis that $\alpha_n$ tends to $\infty$ in Conjecture \ref{BerryConj} with the hypothesis that the sequence $(M_n)$ BS-converges toward $X$. In other words, for every positive real number $R$, if $n$ is large enough, the pointed covering map $(X , x_0 ) \to (M_n , p)$ associated to a random choice of point $p \in M_n$ (and frame in $(X,x_0)$) almost surely maps the $R$-ball $B_X (x_0 , R)$ isometrically onto $B_{M_n} (p,R)$. We shall also always assume that the sequence $(M_n)$ is {\it uniformly discrete}, i.e. that there exists a uniform (in $n$) lower bound on the injectivity radius of $M_n$.

The assumption that $M$ is negatively curved in Conjecture \ref{BerryConj} insures that the geodesic flow is chaotic. In our context we replace this hypothesis by the existence of a uniform spectral gap on the first positive eigenvalue $\lambda_1 (M_n )$. In this case we say that $(M_n)$ forms an \emph{expander family}.

Let us now fix an element $\lambda_0$ in the spectrum of the Laplace operator on $L^2 (X)$. In general $\lambda_0$ is not an eigenvalue of the Laplace operator on $L^2 (\Gamma_n \backslash X)$. However the following Weyl law type result holds (see Lemma \ref{L:deltan}):
\begin{prop} \label{Pintro}
Suppose that $(M_n )$ BS-converges toward $X$. Then there exists a sequence of positive numbers $\delta_n$ that tends to $0$ as $n$ tends to infinity such that
$$\# \{ i \; : \; \lambda_i^{(n)} \in [\lambda_0 -\delta_n , \lambda_0 + \delta_n] \} =  2\delta_n  p (\lambda_0) \mathrm{vol} (M_n ),$$
where $p (\lambda_0 )$ is a positive number\footnote{Here we take $\lambda_0$ in the interior of the spectrum of the Laplacian on $X$ so that in particular $p(\lambda_0) \neq 0$.} --- the density of the spectral measure of the Laplace operator acting on $L^2 (X)$.
\end{prop}
In the rest of this introduction we fix such a sequence $(\delta_n)$ and furthermore assume that $\delta_n \gg 1/\log(\vol (M_n))$. In particular when $n \to +\infty$, the number of eigenvalues in the shrinking interval $[\lambda_0 -\delta_n , \lambda_0 + \delta_n]$ tends to infinity.

Now pick $\phi_j^{(n)}$ with $\lambda_j^{(n)} \in [\lambda_0 - \delta_n , \lambda_0 + \delta_n]$ and lift it using the pointed covering map $(X , x_0 ) \to (M_n , p)$ associated to a random choice of point $p \in M_n$ (and frame in $(X,x_0)$). It yields a random function on $X$ or a random field that we can loosely think of as a probability measure on $C^{\infty} (X)$. This random function is a random translate of $\phi_j^{(n)}$ by an element of $G$ and as such is periodic with respect to a conjugate of $\Gamma$. We shall say that a general random field on $X$ is \emph{aperiodic} if it almost surely admits no non trivial period.
By analogy with Berry's conjecture, we propose the following provocative conjecture. It is inspired by deep recent results of Backhausz and Szegedy \cite{BS} on the distribution of eigenvectors of random regular graphs. 

\begin{conjecture} \label{BC2}
Let $(M_n = \Gamma_n \backslash X )$ be a uniformly discrete expander family that BS-converges toward $X$. Let $\delta_n$ be as in Proposition \ref{Pintro}. Then 
$$\{ (M_n ,  \phi_j^{(n)}) \; : \;  \lambda_j^{(n)} \in [\lambda_0 - \delta_n , \lambda_0 + \delta_n] \}$$
is relatively compact in BS-topology, accumulation points are random fields on $X$ and the only possible aperiodic accumulation point is the isotropic monochromatic Gaussian random wave on $X$ with eigenvalue $\lambda_0$.
\end{conjecture}
Note that one can not hope in general that $(M_n ,  \phi_j^{(n)})$ BS-converges to the isotropic monochromatic Gaussian random wave on $X$ with eigenvalue $\lambda_0$: if the $M_n$'s form a tower of finite coverings and the $\phi_j^{(n)}$'s are just lifts of the very same function $\phi_0$, the limit random field is supported on random lifts of $\phi_0$, that are all periodic functions. 

\begin{unremark} 
Suppose that the sequence $\mathrm{vol} (M_n )$ tends to infinity. It then follows from \cite{7samurai}, and a recent announcement by Fraczyk and Raimbault based on \cite{Fraczyk}, that if either
\begin{itemize}
\item the group $G$ is of real rank $2$ and has Kazhdan's property (T) --- e.g. if $G = \mathrm{SL}_N (\mathbf{R})$ with $N \geq 3$, or
\item the discrete groups $\Gamma_n$ are congruence subgroups,
\end{itemize} 
then $(M_n )$ BS-converges toward $X$ and that the first positive eigenvalue $\lambda_1 (M_n )$ is uniformly bounded away from $0$. In these cases it is moreover expected --- a conjecture due to Margulis \cite[page 322]{margulis:book} --- that the family $(M_n)$ is always uniformly discrete, so that all the hypotheses of Conjecture \ref{BC2} should be satisfied. 
\end{unremark} 

\medskip

Conjecture \ref{BC2} may be hard to prove. However it suggests interesting results that appear to be more tractable. 

\subsection*{A deterministic viewpoint: Quantum Ergodicity}

As already mentioned, rigorous proofs of Berry's phenomenon are missing. However, a consequence of this expected randomness and one of the rare general results known is the Quantum Ergodicity (QE) theorem (due to Shnirelman, Zelditch and Colin de Verdi\`ere \cite{Sni, Zel87, CdV85}). 

Let $M$ be a compact Riemannian manifold with normalized volume form $d\mathrm{vol}_M$. To any norm $1$ function $\phi \in
L^{2}(M)$ corresponds a probability measure $\mu_{\phi}$ on $M$ whose density is $|\phi|^{2}$. 
The Quantum Ergodicity Theorem states that if the geodesic flow is ergodic, then for any fixed orthonormal basis 
$(\phi_{j})$ of eigenfunctions of the Laplace operator, with nondecreasing sequence of
eigenvalues $\lambda_0 \leq \lambda_1 \leq \ldots \leq \lambda_j \to +\infty$, and for any continuous function $a : M \to \R$ we have:
\[
\lim_{N\rightarrow\infty}\frac{1}{N}\sum\limits_{j=1}^{N}\left\vert \int_{M}a \,d\mu_{\phi_j}%
-\int_{M}a \, d\mathrm{vol}_M\right\vert^2 = 0.%
\]
This implies that for a density $1$ subsequence $j_{k}$ of the integers, we
have
\begin{equation}\label{e:QEsubseq}
\lim_{k\rightarrow\infty}\mu_{\phi_{j_{k}}}=d\mathrm{vol}_M
\end{equation}
in weak-$\ast$ convergence. In other words, the densities of probability measures associated with eigenfunctions become uniform in the large eigenvalue limit for most eigenfunctions. The Quantum Unique Ergodicity (QUE) conjecture, due to Rudnick and Sarnak \cite{RS94}, predicts that if $M$ has negative curvature, the full sequence in \eqref{e:QEsubseq} should converge. Although great progress have been made by Lindenstrauss \cite{Lin06}, Anantharaman \cite{Ana08} and Dyatlov-Jin \cite{DJ} in this direction, the conjecture is still open.

It is known that the Gaussian moment conjecture (for fourth moments) implies QUE. On the other hand, it can be deduced from the exceptional behaviour of $L^\infty$-norms of some eigenfunctions in arithmetic hyperbolic manifolds \cite{RS94} that the Gaussian moment conjecture is false in general for large moments. Thanks to the nature of BS sampling, these wild behaviours do not seem to affect our Conjecture \ref{BerryConj}. Note however that Conjecture \ref{BerryConj} is probably hard to prove. In \S \ref{S:relQUE} we indeed prove:

\begin{theorem} \label{weakQUE}
Conjecture \ref{BerryConj} implies that the sequence $(\mu_{\phi_j})$ weakly converges toward $d\mathrm{vol}_M$. 
\end{theorem}

Note that for Theorem \ref{weakQUE}, we do not use that the limiting invariant
random wave is Gaussian, just that it is ergodic and does not lose
energy. The energy stability implies that any subsequential limit of
square measures is absolutely continuous with respect to volume and an
ergodicity argument then shows that the limiting measure must be equal
to the volume. This is related to an observation of Hejhal and Rackner \cite[\S 5.1]{Hejhal}.

Conjecture \ref{BC2} suggests considering a different point of view where one takes eigenfunctions in a fixed spectral window and varies instead the manifold. Keeping notations as in Conjecture \ref{BC2}, we expect the correlation function $\phi_j^{(n)} (x) \phi_j^{(n)} (y)$ to be close to the correlation function of the isotropic monochromatic Gaussian random wave on $X$ with eigenvalue $\lambda_0$. We shall see that, by definition, the latter is the spherical function $\varphi_{\lambda_0} (x,y)$ on $X$ that only depends on the distance $d(x,y)$ and is a $\lambda_0$-eigenfunction of $y$ when $x$ is fixed. 

We address this question through a sequence of $\Gamma_n$-invariant test kernels on $X$: 
\begin{equation} \label{intro:test}
A^{(n)} : \Gamma_n \backslash (X \times X) \to \R \mbox{ with } A^{(n)} (x,y) =0 \mbox{ if } d(x,y) > M
\end{equation}
for some uniform (in $n$) constant $M$. The kernel $A^{(n)}$ defines an operator $\mathbf{A}^{(n)}$ on $C^{\infty} (M_n )$ by the formula 
$$(\mathbf{A}^{(n)} f ) (x) = \int_X  A^{(n)} (x,y) f(y) dy \quad \left( f \in C^\infty (M_n ) \right).$$ 
Finally, the expression $ \int_X A^{(n)} (x,y) \varphi_{\lambda_0} (x,y) dx$ being $\Gamma_n$-invariant we may define
$$\langle \mathbf{A}^{(n)} \rangle_{\lambda_0} = \int_{M_n} \left( \int_X A^{(n)} (x,y) \varphi_{\lambda_0} (x,y) dx \right) d\mathrm{vol}_{M_n}(y).$$ One would like to prove that  
$$\int_{M_n} \phi_j^{(n)} (x)  (\mathbf{A}^{(n)} \phi_j^{(n)} ) (x)  d\mathrm{vol}_{M_n}(x) - \langle \mathbf{A}^{(n)} \rangle_{\lambda_0}$$
tends to $0$ as $n$ tends to infinity. This is not true in general because of the existence of periodic subsequences $(M_n ,  \phi_j^{(n)})$. However, when the real rank of $X$ is $1$, we prove the following analogue of QE in this context.

\begin{theorem} \label{T:introQE}
Suppose that $\mathrm{rank}_\R (X) = 1$. Let $(M_n )$ be a uniformly discrete expander family that BS-converges toward $X$ and let $(A^{(n)})_{n \in \mathbb{N}}$ be a uniformly bounded sequence of test kernels satisfying \eqref{intro:test}. There exists a sequence $(\delta_n )$ as in Proposition \ref{Pintro} such that, setting $I_n = [\lambda_0 - \delta_n , \lambda_0 + \delta_n]$ and letting 
$$N(\delta_n , \Gamma_n ) = \# \{ j \; : \; \lambda_j^{(n)} \in I_n \},$$ 
we have:
$$\frac{1}{N(\delta_n , \Gamma_n )} \sum_{\lambda_j^{(n)} \in I_n} \left| \int_{M_n} \phi_j^{(n)} (x)  (\mathbf{A}^{(n)} \phi_j^{(n)} ) (x) d\mathrm{vol}_{M_n}(x) - \langle \mathbf{A}^{(n)} \rangle_{\lambda_0} \right|^2 \to 0$$
as $n$ tends to infinity. 
\end{theorem}

Note that here the sequence $(\delta_n)$ depends on the sequence of test kernels. In the remark following Theorem \ref{T2} we formulate a version of this theorem for a window of (uniform) positive radius $\delta$. A result of this nature was first obtained recently by Le Masson and Sahlsten \cite{LMS} for sequences of hyperbolic surfaces. Our result gives in particular a generalization to hyperbolic manifolds of any dimension. Note also that in \cite{LMS}, only operators obtained by multiplication by functions are considered, and the spectral window is not authorized to shrink. Our proof is similarly based on the use of the mixing dynamics in the form of an ergodic theorem of Nevo, and is developed in Sections \ref{s:qes1}, \ref{s:qes2} and \ref{s:qes4}. The work of \cite{LMS} and ours are deeply inspired by results on large regular graphs by Anantharaman and Le Masson \cite{ALM15} and a variation of the proof appearing in \cite{BLML15}. On discrete regular graphs the spectrum of the Laplacian is always bounded and the relevant limit becomes that of large graphs. 
One of the advantages of the Benjamini-Schramm formalism is to unify these theories by providing the same framework for the discrete and continuous, large eigenvalue and large volume settings.

\subsection*{A random viewpoint}

Proving that the limits associated with deterministic eigenfunctions in Conjecture \ref{BerryConj} and Conjecture \ref{BC2} exhibit Gaussian behavior is most likely a very difficult problem. However we show that for random superpositions of the eigenfunctions $\phi_j^{(n)}$ Conjecture \ref{BC2} holds. In fact we prove a stronger result, see Theorem \ref{I:T5} below. 

Considering random superpositions naturally leads to two processes on measures on $C^\infty (X)$: we may either first pick a random superposition and then lift it to $X$ via a random projection $X \to M_n$, or first pick a random projection and then lift a random superposition using this projection. In other words for the first process we first choose a random superposition, and then look at the function we just obtained around a randomly chosen point (and frame). In the second case we first choose a point at random and look at a random superposition around this point. This leads to two different random {\it processes on processes} on $C^\infty (X)$, or equivalently to two measures on $\mathbb{M}^1(C^\infty(X))$, the space of probability measures on $C^\infty (X)$ that we denote respectively by $\alpha$ and $\beta$. In general, the process associated with $\alpha$ is richer as far as we are interested in random eigenfunctions. For a precise definition of these two processes in the level aspect, see Section \ref{s:randomization}.

In the context of Conjecture \ref{BerryConj}, Nazarov and Sodin have shown that for any base point $x$ in $M$ picking a random element in the unit sphere of the finite dimensional space spanned by eigenfunctions of eigenvalues $\leq r$ defines --- after rescaling by $r$ --- a process $f_{x,r}$ which is close in law with the process $F_x$ associated to the monochromatic Gaussian random Euclidean wave with eigenvalue $1$. In particular the process $\beta$ above converges in law to the Dirac measure on the monochromatic Gaussian random Euclidean wave with eigenvalue $1$. In fact Nazarov and Sodin prove a stronger coupling result: they show that one can find a coupling $(f_{x,r} ' , F_x ')$ of random variables defined on the same probability space such that $f_{x,r}$ and $f_{x,r} '$, resp. $F_x$ and $F_x '$, have the same law, and with high probability $f_{x,r} '$ and $F_x '$ are close in $C^1$-norm. See \cite[\S 2.2]{SBM} or \cite[Example 1.1 and p. 1116-27]{NaliniBourbaki}. This step is crucial in the work of Nazarov and Sodin \cite{NS,SBM} on the asymptotic counting of nodal domains. 

Our second theorem is related to the (weak form of the) result of Nazarov and Sodin but in the level aspect and for the richer process $\alpha$.

\begin{theorem} \label{I:T5}
Let $(M_n )$ be a uniformly discrete sequence that BS-converges toward $X$. There exists a sequence $(\delta_n)$ as in Proposition \ref{Pintro} such that if for each $n$, we pick uniformly at random a function $\phi_n$ in the unit sphere of 
$$\mathrm{span} \{ \phi_j^{(n)} \; : \; \lambda_j^{(n)} \in [\lambda_0 - \delta_n , \lambda_0 + \delta_n ] \}$$
and randomly lift it to $X$, then the resulting process on $\mathbb{M}^1(C^\infty(X))$ converges in law to the Dirac measure on the isotropic monochromatic Gaussian random wave on $X$ with eigenvalue $\lambda_0$.
\end{theorem}
Note that here again the sequence $(\delta_n)$ depends on $(M_n)$, see Lemma \ref{L:T1abis}.

This theorem is a consequence of Theorem \ref{T1}, see Section \ref{s:randomization}. Note that in this situation we do not need the hypothesis that $(M_n )$ is an expander family. 

One of the main points of this theorem is to go from the convergence of the process $\beta$, which is equivalent to a local Weyl law type of result (see Lemma \ref{L:T1abis} and its proof), to the convergence of the process $\alpha$. This is done using the ergodicity of the Gaussian random wave (Proposition \ref{L:Erg}). We expect that a similar argument can be used to prove an analogue of Theorem \ref{I:T5} in the eigenvalue aspect.

The use of ergodicity to go from a local Weyl law result to a result about \emph{almost all} eigenfunctions (or more precisely here random superpositions) is reminiscent of the phenomenon underlying the quantum ergodicity theorem. Let us recall the heuristic of the proof in the large eigenvalue limit, using the same notation as Theorem \ref{weakQUE}. We denote by $\rho_j$ the microlocal lift of the eigenfunction $\phi_j$ (a probability measure on the unit cotangent bundle $S^*M$ that projects to $\mu_{\phi_j}$ and is asymptotically invariant under the geodesic flow). The local Weyl law says that 
$$ \frac1{N} \sum_{j=1}^N \rho_j \to \omega$$
weakly when $N \to +\infty$, where $\omega$ is the Liouville measure on $S^*M$. By ergodicity of the Liouville measure, it cannot be decomposed as a (finite) convex combination of invariant measures unless they are all equal. This implies that almost every term $\rho_j$ in the sum tends to $\omega$ when $N\to +\infty$ (see \cite{Zel06} for more details on this point of view).

In our case, the lifts considered are given by the BS-sampling and the ergodicity is that of the Gaussian wave. We use a local Weyl law type of argument to show that  $\beta$ converges to the Dirac mass at the Gaussian wave. We then remark that the expected values $\mathbb{E}(\alpha)$ and $\mathbb{E}(\beta)$ are equal, and deduce that $\mathbb{E}(\alpha)$ converges to the Gaussian wave. Seeing the expected value as a convex combination of invariant measures, and using the ergodicity of the Gaussian wave, the limit of $\alpha$ has to be equal to the Dirac mass at the Gaussian wave.

\medskip

As we have tried to emphasise, Benjamini-Schramm convergence provides a unified language for both the level and eigenvalue aspects of Berry's conjecture. In fact, we believe that one of the main interest of the Benjamini-Schramm viewpoint is to naturally lead to many interesting questions and new results. Theorem \ref{I:T5} is an example of such a new result: the formulation of the statement entirely relies on the idea of BS-sampling. We conclude the article by a list of questions suggested by the Benjamini-Schramm viewpoint.

\bigskip

The organization of the paper is as follows. In Section \ref{S:BS} we introduce the notions of (decorated) Benjamini-Schramm convergence on general Riemannian manifolds, and of Invariant Random Subgroups in the case of symmetric spaces. We then define the Gaussian random wave and give some of its properties in Section \ref{S2}. We state and discuss the random wave conjectures in Section \ref{S:berry}, where we also prove that Berry's conjecture implies QUE in the eigenvalue aspect (the content of Theorem \ref{weakQUE}). Sections \ref{s:qes1}, \ref{s:qes2} and \ref{s:qes4} develop the proof of the Quantum Ergodicity theorem in the level aspect (Theorem \ref{T:introQE}). Theorem \ref{I:T5} about random superpositions of eigenfunctions is proved in Section \ref{s:randomization}. Finally, in Section \ref{s:problems} we list some questions and open problems. 

\subsection*{Acknowledgments} We thank Nalini Anantharaman, Farrell Brumley, Alix Deleporte, Maxime Ingremeau, Bart Michels, Mostafa Sabri, Roman Schubert, Nicolas Tholozan and Joe Thomas for interesting comments regarding the topics of this paper. Our deep gratitude goes to an anonymous referee who carefully read our paper and made us numerous corrections and suggestions. \\
M.A. was supported by the ERC Consolidator Grant 648017 and the NKFIH grant K109684. E.L.M. was partially supported by the Marie Sk{\l}odowska-Curie Individual Fellowship grant 703162 while at the University of Bristol, by a Rutherford fellowship at the University of Warwick, and by Initiative d’Excellence Paris//Seine.

\section{Benjamini-Schramm convergence and Invariant Random Subgroups} \label{S:BS} 

Let us start by giving an imprecise but quick explanation on the core notion
of Benjamini-Schramm convergence of finite volume Riemannian manifolds and
also on BS limits of a sequence of finite volume manifolds endowed with
smooth functions. We say that two large volume manifolds $M$ and $N$ are BS close, if for some large radius $R>0$, the distribution of $R$-balls centered at a volume-random point of $N$ and a volume-random point of $M$,
are close. This means that by picking independently randomly a large number
of points in $M$ and $N$, considering the $R$-balls around these points as
metric spaces, and allowing small distortions in the geometry, we can not
statistically distinguish $M$ and $N$. For instance, if both $M$ and $N$ are
finite volume quotients of the same hyperbolic space $X$, and their infimal 
injectivity radius is large as well, then these balls will all be isometric to the corresponding ball in $X$, hence they are close.

This distance notion can be made precise, and then we say that $M_{n}$ is BS
convergent if it converges in this metric. The limiting object is a random
rooted manifold, which can be best seen from the abstract (and precise)
definition below. 

When we also put smooth functions $\phi_{n}:M_{n}\rightarrow R$, the notion of
BS convergence stays the same, but we also look at the decoration of the $R$-balls around random points of $M_{n}$ coming from $\phi_{n}$. So, the limit
will be a decorated random rooted manifold. When the manifolds $M_{n}$
without decoration BS converge to a fixed homegeneous space $X$, it is the
functions $\phi_{n}$ that carry the interesting information on $M_{n}$. In this
case the limit of the manifolds will be the fixed $X$, and $\phi_{n}$ will turn
into a random function on $X$, that, by nature, will be invariant to
isometries of $X$, in distribution.

The limiting functions will sense how $\phi_{n}$ looked locally from a random point of $M_{n}$. We lose some
properties in this limiting process, e.g. the limit of $L^{2}$ functions
with norm $1$, will not be $L^{2}$ in the limit, it will be an invariant
random wave on $X$ and the role of the $L^{2}$ norm will be taken by the
expected value of the square of the value of this wave at the origin of $X$.

Let us now give precise definitions. Consider the space $\mathcal{M}^d$ of pointed, connected, complete Riemannian manifolds of dimension $d$ up to pointed isometries, with 
its {\it smooth topology,} see \cite[\S A.1]{AbertBiringer}.  In this topology, two pointed manifolds $(M,p)$ and $(N,q)$ are close if there are compact subsets of $M$ and $N$ containing large radius neighborhoods of the base points that are diffeomorphic via a map that is $C^\infty$-close to an isometry. More precisely: a sequence of pointed Riemannian manifolds $(M_n , p_n)$ $C^\infty$-converges toward $(M, p)$ if for every radius $R>0$, there is a sequence of maps $f_n : B_M (p,R) \to M_n$ with $f_n (p) = p_n$ such that the Riemannian metric $f^* m_n$ on the metric ball $B_M (p , R)$ inside $M$, pulled back from the Riemannian metric $m_n$ on $M_n$, converges to the restriction of the Riemannian metric $m$ on $M$ in $C^\infty$-topology. It can be proved that the smooth topology on $\mathcal{M}^d$ is induced by a Polish topology, i.e. the space $\mathcal{M}^d$ is separable and completely metrizable; see \cite[\S A.2]{AbertBiringer} or \cite{CandelAlvarezLopezBarralLijo} for a proof of this result; we elaborate on this below. The space $\mathcal{M}^d$ is not compact but Cheeger's compactness theorem implies that the subspace that consists of pointed manifolds $(M,p)$ with uniformly bounded geometry is a compact subspace; see \cite[\S A.1]{AbertBiringer}. 

Suppose $M$ is a compact connected complete Riemannian $d$-manifold. Pushing forward the normalized Riemannian volume measure under the map
$$M \to \mathcal{M}^d , \ p \mapsto (M,p)$$
one obtains a probability measure $\mu_M$ on $\mathcal{M}^d$. Let us observe that the measures thus obtained are particular.

Let $T^1 \mathcal{M}^d$ be the space of isometry classes of rooted unit tangent bundles $(T^1 M , p ,v)$, where $v \in T_p^1 M$. The geodesic flows on individual $T^1 M$ combine to give a continuous flow 
\begin{equation} \label{geodflow}
\mathbf{g}_t : T^1 \mathcal{M}^d \to T^1 \mathcal{M}^d.
\end{equation}
On the other hand, each fiber $T_p^1 M$ of 
$$T^1 \mathcal{M}^d \to \mathcal{M}^d; \ (M,p,v) \mapsto (M,p)$$
comes with a (Liouville) measure $\omega_{M,p}$ induced by the Riemannian metric on $M$. Any measure $\mu$ on $\mathcal{M}^d$ can then be lifted to a measure $\widetilde{\mu}$ on $T^1 \mathcal{M}^d$ defined by the equation $d\widetilde{\mu} = \omega_{M,p} d\mu$. The Liouville measure on the unit tangent bundle of a Riemannian manifold is invariant under the geodesic flow and, similarly, the measure $\widetilde{\mu}_M$ is invariant under the flow $\mathbf{g}_t$. 

Following \cite{AbertBiringer}, we say that a measure $\mu$ on $\mathcal{M}^d$ is {\it unimodular} if $\widetilde{\mu}$ is invariant under \eqref{geodflow}. We refer to \cite{AbertBiringer} for other characterizations of unimodularity.
 
\begin{definition} \label{def7}
A sequence $(M_n )$ of compact connected complete Riemannian $d$-manifolds is convergent in the sense of Benjamini-Schramm, or just \emph{BS-converges}, if the sequence $\mu_{M_n}$ converges in the weak* topology of the set of all unimodular probability measures on $\mathcal{M}^d$.
\end{definition}
Recall that a sequence of probability measures $(\mu_n )$ on $\mathcal{M}^d$ converge to $\mu$ in the {\it weak* topology} if $\int f d \mu_n \to \int f d\mu$ for every bounded, continuous function $f : \mathcal{M}^d \to \R$ (beware that some authors refer to this topology as {\it weak topology}).

\medskip
\noindent
{\it Example.} We will mostly deal with the very particular case where $(M_n)$ is a sequence of compact quotients of a given homogeneous space $X$; e.g., take $X$ to be the hyperbolic plane $\mathbf{H}$ and $(M_n)$ to be a sequence of closed connected hyperbolic surfaces. Since $X$ is homogeneous the pointed space $(X,p)$ is independent of the choice of $p$ and the Dirac measure $\delta_X$ at $(X,p)$ defines a unimodular probability measure on $\mathcal{M}^d$. One can prove that the sequence $(M_n)$ BS-converges toward $\delta_X$ if for every $R>0$, the probability that the $R$-ball centered a random point in $m_n$ is isometric to the $R$-ball in $X$ tends to $1$ when $n \to \infty$; i.e., for every $R>0$, we have 
$$\lim_{n \to \infty} \frac{\vol ((M_n)_{<R}) }{\vol (M_n )} = 0,$$ 
where $M_{<R} = \{ x \in M \; : \; \mathrm{InjRad}_M (x) < R \}$ is the $R$-thin part of $M$; see \cite[Corollary 3.8]{7samurai}. In that case we loosely say that $(M_N)$ BS-converges toward $X$. 

A well-studied example is when $M_n = \Gamma_n \backslash X$ where the groups $\Gamma_n$ form a chain of normal subgroups with trivial intersection in some fixed uniform
lattice of the isometry group of $X$; in this case, the $R$-thin part of $M_n$ is empty for large enough $n$.

\medskip

In general the geometry around a randomly chosen point in $M_n$ has no reason to be closer and closer to a given homogeneous manifold. The limiting object is then a unimodular probability measure on $\mathcal{M}^d$ that precisely encodes what the geometry of $M_n$, for large $n$, looks like near randomly chosen base points. This perspective is elaborated on in Section 3 of \cite{7samurai} and studied in great details in \cite{AbertBiringer}. Elaborating on the compacity theorem of Cheeger already alluded to, it can be proved that the set of all unimodular probability measures on $\mathcal{M}^d$ that are concentrated on pointed manifolds with pinched negative curvature and uniform upper and lower bounds on all derivatives of curvature is weak* compact; see \cite[Theorem 1.10]{AbertBiringer}. We will not use this result. 

\medskip 

One similarly defines the BS-convergence of a sequence $(M_n)$ with functions $\phi_n: M_n \to \R $: consider the space 
$$\mathcal{E}^d = \left\{ (M, p , \phi) \; \left| \;  \begin{array}{l} 
M \mbox{ connected, complete Riemannian $d$-manifolds}, \\ p \in M, \ \phi : M \to \R \mbox{ smooth} \end{array} \right. \right\} \Big/ \begin{array}{l} \mbox{pointed} \\ \mbox{isometries} \end{array}$$  
equipped with its smooth topology where $[M,p, \phi]$ is close to $[N,q, \psi]$ if there are compact subsets of $M$ and $N$ containing large radius neighborhoods of the base points that are diffeomorphic via a map $f$ that is $C^\infty$-close to an isometry and s.t. $\phi$ and $\psi \circ f$ are $C^\infty$-close. 

\begin{prop}
The topological space $\mathcal{E}^d$ has a compatible structure of a Polish space (a complete, separable metric space).
\end{prop}
\begin{proof}
We elaborate on the proof of \cite[Theorem A.9]{AbertBiringer}. There for each positive real $R$ and each integer $k \in \mathbf{N}$ a function 
$d_{R,k} : \mathcal{M}^d \times \mathcal{M}^d \to  \R$ is defined by 
$$d_{R,k} ((M,p) , (N,q)) = \inf \{ \log \lambda \; : \; (N,q) \in \mathcal{N}^k_{R/\lambda , 1/ \lambda, \lambda} (M,p) \}$$
where $\mathcal{N}^k_{R, r, \lambda} (M,p)$ is the set of all $(N,q)$ such that there is a smooth embedding
$$f : B_M (p , R)  \to N $$
with $f(p) = q$ such that the iterative total derivative $D^k f$ is locally $\lambda$-bilipschitz on some $r$-neighborhood of $B_M (p,R)$ in the $k$-fold iterated tangent bundle of $M$. Each $d_{R,k}$ satisfies an (asymmetric) triangle inequality.
Now given a triple $(M,p,\phi)$ that represents a point in $\mathcal{E}^d$ we define $\mathbf{N}^k_{R , r,  \lambda} (M,p)$ to be the set of all $(N,q)$ such that there is a smooth embedding
$$f : B_M (p , R)  \to N $$
with $f(p) = q$ such that the iterative total derivative $D^k f$ is locally $\lambda$-bilipschitz on some $r$-neighborhood $Z$ of $B_M (p,R)$ in the $k$-fold iterated tangent bundle of $M$ and 
$$||D^k (\phi ) - D^k (\psi \circ f) ||_{L^\infty (Z)} \leq \log \lambda.$$
We then define a function $\mathbf{d}_{R,k} : \mathcal{E}^d \times \mathcal{E}^d \to  \R$ by 
$$\mathbf{d}_{R,k} ((M,p , \phi ) , (N,q, \psi )) = \inf \{ \log \lambda \; : \; (N,q) \in \mathcal{N}^k_{R/\lambda , 1/ \lambda,  \lambda} (M,p , \phi ) \}.$$
The proof that $d_{R,k}$ satisfies an (asymmetric) triangle inequality and the triangle inequality for $|| \cdot ||_\infty$ imply that $\mathbf{d}_{R,k}$ also  satisfies an (asymmetric) triangle inequality. 

The subsets of $\mathcal{E}^d$ defined for each triple $(M,p,\phi)$ that represents a point in $\mathcal{E}^d$, $R>0$, $k \in \mathbf{N}$ and $\varepsilon >0$ by 
$$\mathbf{d}_{R,k} ((M,p , \phi ) ,  \cdot ) < \varepsilon$$
form a basis of the topology on $\mathcal{E}^d$. The maps $d_{R,k}$ are not symmetric but the reversed inequalities 
$$\mathbf{d}_{R,k} ( \cdot , (M,p , \phi )) < \varepsilon$$
define a basis for the same topology according to \cite[Lemma A.2]{AbertBiringer}. The function 
$$\mathbf{D} :\mathcal{E}^d \times \mathcal{E}^d \to  \R; \quad \mathbf{D} (x , y  ) = \sum_{k=1}^{+\infty} 2^{-k} \min \{ \mathbf{d}_{k,k} (x,y) , 1 \}$$
then defines a metric on $\mathcal{E}^d$ that induces the topology defined above. 

We now show that $\mathcal{E}^d$ is separable. Since any element of $\mathcal{E}^d$ is a limit of triples $(M,p,\phi)$ where $M$ is a closed Riemannian manifold, it suffices to construct a countable subset of $\mathcal{E}^d$ that accumulates onto every point $[M,p,\phi ] \in \mathcal{E}^d$ with $M$ closed.
 
Now Cheeger compactness theorem implies that there are only countably many diffeomorphism types of closed manifolds so we can work with the subspace of points $[M,p,\phi ] \in \mathcal{E}^d$ with $M$ closed and in a fixed diffeomorphism class. This subspace is the product of $M$ with the space of Riemannian metrics on $M$ and $C^\infty (M)$ (both equipped with the smooth topology). Separability then follows from the separability of $M$ and Weierstrass approximation theorem. 

It remains to show that $(\mathcal{E}^d , \mathbf{D})$ is complete. This follows from a simple adaptation of the proof of \cite[Theorem A.9]{AbertBiringer}.
\end{proof}

\begin{unremark} 
One could similarly prove that the subspace of $\mathcal{E}^d$ that consists of triplet $[M,p, \phi]$ where $(M,p)$ has uniformly bounded geometry and $\phi$ has all its derivatives uniformly bounded from below and from above is compact. 
\end{unremark}

\medskip

We equip the topological space $\mathcal{E}^d$ with the $\sigma$-algebra $\mathcal{B}$ generated by its open sets. A probability measure on $\mathcal{E}^d$ is a $\sigma$-additive function $\mathcal{B} \to [0,1]$ that maps the entire space $\mathcal{E}^d$ to $1$. Recall that a sequence of probability measures $\mu_n$ on $\mathcal{E}^d$ is said to converge in the weak* topology toward a probability measure $\mu$ if for each bounded, continuous real function $f$ on $\mathcal{E}^d$ we have 
$$\lim_{n \to \infty} \mu_n (f) = \mu (f).$$
In practice another characterisation of weak* convergence is useful: since $\mathcal{E}^d$ is a Polish space, the portmanteau theorem \cite[Theorem 2.1]{Billingsley} implies that $(\mu_n )$ weak* converges to $\mu$ if and only if for all continuity sets $A$ of $\mu$ --- i.e. a Borel set whose boundary set has $\mu$-measure $0$ --- we have 
$$\lim_{n \to \infty} \mu_n (A) = \mu (A).$$
Finally note that since $\mathcal{E}^d$ is a Polish space, Prokhorov's theorem \cite[Section 5]{Billingsley} implies that a collection of probability measures on $\mathcal{E}^d$ is relatively compact w.r.t. the weak* topology if and only if it is tight. 

\medskip

As above, if $M$ is a compact connected complete Riemannian $d$-manifold and $\phi : M \to \R$ a smooth map, pushing forward the normalized Riemannian volume measure under the map
\begin{equation} \label{map:pf}
M \to \mathcal{E}^d , \ p \mapsto [M,p , \phi] 
\end{equation}
one obtains a probability measure $\mu_{M , \phi}$ on $\mathcal{E}^d$. 

As with $\mathcal{M}^d$ we shall denote by $T^1 \mathcal{E}^d$ the space of isometry classes of rooted unit tangent bundles colored with a function $(T^1 M , p , v , \phi)$ where $v \in T_p^1M$. Here again it comes equipped with a continuous (geodesic) flow and any measure $\mu$ on $\mathcal{E}^d$ can be lifted to a measure $\widetilde{\mu}$ on $T^1 \mathcal{E}^d$ using the volume form on the fiber. The measure $\widetilde{\mu}_{M_n , \phi_n}$ is not invariant under the geodesic flow unless $\phi_n$ is constant on $M_n$. In Section \ref{S:berry} we consider measures on $\mathcal{E}^d$ whose lifts on $T^1 \mathcal{E}^d$ are invariant under the geodesic flow.

\begin{definition}  \label{def8}
A sequence $(M_n , \phi_n )$, where each $M_n$ is a compact connected complete Riemannian $d$-manifolds and $\phi_n : M_n \to \R$ is smooth, is convergent in the sense of Benjamini-Schramm, or just \emph{BS-converges}, if there exists a probability measure $\mu$ on $\mathcal{E}^d$ such that the sequence $\mu_{M_n , \phi_n}$ converges to $\mu$ in the weak* topology.
\end{definition}

\medskip

In this paper we shall largely focus on sequences of manifolds modeled on a given symmetric space $X=G/K$ as defined in the Introduction, i.e. $X$-manifolds $\Gamma \backslash X$ where $\Gamma \subset G$ is discrete and torsion free. 

Let $\mathrm{Sub}_G$ be the space of closed subgroups of $G$, endowed with its Chabauty topology, see \cite{7samurai}. 

\begin{definition}
An \emph{invariant random subgroup} (IRS) of $G$ is a Borel probability measure $\mu$ on $\mathrm{Sub}_G$ that is invariant under the conjugation action of $G$ on $\mathrm{Sub}_G$. 
\end{definition}

These were first studied in \cite{7samurai}. An important family of IRSs are associated to lattices in $G$. Suppose in particular that $\Gamma \subset G$ is a uniform, torsion free, discrete subgroup of $G$. Pushing forward the invariant probability measure on $G / \Gamma$ by the stabilizer map for the left  action of $G$ on $G / \Gamma$, indeed yields an IRS $\mu_\Gamma$. This is very similar to the construction of $\mu_M$ above, and more generally, there is a dictionary between IRSs of $G$ and unimodular random $X$-manifolds: the map $\Gamma \mapsto (\Gamma \backslash X, \Gamma eK)$ induces a weak*-homeomorphism between the spaces of distributions of discrete, torsion free IRSs of $G$ and of unimodular random $X$-manifolds. In particular a sequence of compact $X$-manifolds $(\Gamma_n \backslash X)$ BS-converges if, and only if, the sequence $(\mu_{\Gamma_n})$ of probability measures on $\mathrm{Sub}_G$ converges in the weak* topology. We refer to \cite[Corollary 3.8]{7samurai} for more details; see also \cite{AbertBiringer}.

\medskip
{\it Examples.} We have recalled in the example after Definition \ref{def7} that if $\Gamma_n$ is a sequence of finite index normal subgroups in a fixed uniform lattice $\Gamma \subset G$:
$$\cdots \Gamma_{n+1} \triangleleft \Gamma_n \triangleleft \cdots \triangleleft \Gamma \quad \mbox{such that} \quad \bigcap_n \Gamma_n = \{ 1 \},$$
then the sequence of compact $X$-manifolds $(\Gamma_n \backslash X)$ BS-converges toward $X$. In terms of IRSs this is equivalent to the fact that the sequence $(\mu_{\Gamma_n})$ of probability measures on $\mathrm{Sub}_G$ converges in the weak* topology toward the Dirac measure supported on the trivial subgroup of $G$. We refer to the latter as the {\it trivial IRS} of $G$. 

When $G$ is of real rank $\geq 2$ and has property (T) --- e.g. if $G = \SL_3 (\R)$ --- it follows from Corollary 4.7 in \cite{7samurai} that for any sequence $(\Gamma_n ) \subset \mathrm{Sub}_G$ such that $\vol (\Gamma_n \backslash X)$ tends to infinity, the sequence $(\mu_{\Gamma_n})$ converges in the weak* topology toward the trivial IRS of $G$. 

This is not true in general --- e.g. if $G = \SL_2 (\R)$, one can model hyperbolic surfaces along graphs with many short loops and construct counterexamples. However in general for any sequence $(\Gamma_n )$ of \emph{congruence} subgroups of $G$ such that $\vol (\Gamma_n \backslash X)$ tends to infinity, the sequence $(\mu_{\Gamma_n})$ converges in the weak* topology toward the trivial IRS of $G$, see \cite[\S 5]{7samurai}.

\medskip

Here again we may decorate a manifold $\Gamma \backslash X$ with a function $\phi : \Gamma \backslash X \to \R$. In fact working with IRSs it is more natural to decorate a quotient $\Gamma \backslash X$ with a $\Gamma$-invariant function $\phi : X \to \R$.\footnote{Note that this is stronger, even fixing the base point $\Gamma eK$ there is no canonical lift to $X$ of a function on $\Gamma \backslash X$, one needs to fix a frame.}  We shall therefore rather work with the space of decorated closed subgroups:
$$\widehat{ \mathrm{Sub}}_G  = \{ (H , \phi ) \; : \; H \in \mathrm{Sub}_G \mbox{ and } \phi \in C^{\infty} (X ) \  H\mbox{-invariant} \}$$
with topology induced by the product of the Chabauty topology on $\mathrm{Sub}_G$ and the $C^\infty$-topology on $C^\infty (X)$. 

To $(\Gamma , \phi )$ corresponds the map 
$$G / \Gamma \to \widehat{ \mathrm{Sub}}_G ; \quad g\Gamma \mapsto (g\Gamma g^{-1}  , \phi (g^{-1} \cdot )  ).$$
We shall denote by $\mu_{\Gamma , \phi}$ the push forward of the invariant probability measure on $G / \Gamma$. We may furthermore push forward this measure using the map 
$\widehat{ \mathrm{Sub}}_G \to C^\infty (X)$ and obtain an invariant probability measure $\mu_\phi$ on $C^\infty (X)$. 

Note that if $\phi$ is of norm $1$ on $\Gamma \backslash G$, i.e.
$$\frac{1}{\mathrm{vol} (\Gamma \backslash G)} \int_{\Gamma \backslash G} |\phi|^2 = 1$$
then $\mu_{\Gamma , \phi}$ and $\mu_\phi$ satisfy the following normalisation property:
\begin{equation} \label{energie}
\begin{split}
\int_{\widehat{ \mathrm{Sub}}_G} |\psi (eK) |^2 d\mu_{\Gamma , \phi} (H , \psi) & = \int_{C^\infty (X)} |\psi (eK) |^2 d\mu_{\phi} (H\psi) \\ 
& = \frac{1}{\mathrm{vol} (\Gamma \backslash G)} \int_{G/\Gamma} |\phi (g^{-1})|^2 dg =1.
\end{split}
\end{equation}

\begin{definition} \label{def10}
A sequence $(\Gamma_n \backslash X, \phi_n )$, where each $\Gamma_n \backslash X$ is a compact $X$-manifold and each $\phi_n$ is a smooth $\Gamma_n$-invariant function on $X$, is \emph{weakly convergent}, if there exists a probability measure $\mu$ on $\widehat{ \mathrm{Sub}}_G$ such that the sequence $\mu_{\Gamma_n , \phi_n}$ converges to $\mu$ in the weak* topology. 
\end{definition}
When moreover $(\Gamma_n \backslash X)$ BS-converges to $X$, we will often abusively identify the limit measure with the weak limit $\lim \mu_{\phi_n}$ --- a $G$-invariant measure on $C^\infty (X)$.

\begin{unremark} 
It is natural to expect that if $(\Gamma_n \backslash X, \phi_n )$ is a weakly converging sequence in the sense of Definition \ref{def10} it BS-converges in the sense of Definition \ref{def8} but we have not written a proof. 
\end{unremark}

\section{Gaussian fields on symmetric spaces} \label{S2}

\subsection{Gaussian fields}

Let $M$ be a smooth manifold. A smooth random field\footnote{In most of the paper we shall rather work with the probability measure on $C^\infty (M)$ which associates to a measurable subset $A \subset C^\infty (M)$ the non-negative number $\mathbb{P} (F \in A )$.} $F$ on $M$ is said to be {\it Gaussian} if for every $n$ and for every $n$-tuple of points 
$(x_1 , \ldots , x_n ) \in M^n$, the vector $(F(x_1) , \ldots , F (x_n )) \in \R^n$ has a Gaussian distribution. We refer to \cite{Hida} for details. 

In practice we will only deal with spaces $M=X$ equipped with a transitive action of a Lie group $G$ and we will only consider $G$-invariant random fields.\footnote{Equivalently $G$-invariant measures on $C^\infty (X)$.} We shall say that an invariant Gaussian random field is {\it standard} if for every $x \in X$, we have:
$$\mathbb{E} [F(x)] = 0 \mbox{ and } \mathbb{E} [F(x)^2] =1 .$$
Such random fields are determined by their {\it covariance kernel} 
$$\mathbb{E} [F(x) F(y)].$$

We shall see in the next paragraphs that the relationship with representation theory naturally leads to consider covariance function with complex values. Say that a {\it standard complex Gaussian field on $X$} is a complex-valued random field $F$ on $X$ whose whose real and imaginary parts are independent, identically distributed real Gaussian field and such that 
\begin{enumerate}
\item the map $(x,y) \mapsto \mathbb{E} [F(x) F(y)]$ is identically zero, and
\item the {\it covariance kernel} $(x,y) \mapsto \mathbb{E} [F(x) \overline{F(y)}]$ is constant, equal to $1$, on the diagonal $x=y$.
\end{enumerate}

\medskip
\noindent
{\it Example.} The {\it invariant (or isotropic) monochromatic Gaussian random Euclidean wave with eigenvalue $\mu^2$} is the standard complex Gaussian random field $F_{{\rm unif}, \mu} : \R^d \to \C$ whose covariance kernel is 
$$\mathbb{E} [F_{{\rm unif}, \mu} (x) F_{{\rm unif}, \mu} (y)] = \int_{\mathbf{S}^{d-1}} e^{i \mu \langle y-x , \xi \rangle} d \xi$$
where $\langle , \rangle$ is the standard scalar product on $\R^d$. 
In dimension $2$, one may also describe $F_{{\rm unif}, \mu}$, in polar coordinates, as 
$$F_{{\rm unif}, \mu} (r,\theta ) = \sum_{n \in \Z} c_n J_{|n|} (\mu r) e^{in \theta }$$
where $J_n$ is the $n$-th Bessel function, and $(c_n )_{n \in \Z}$ are standard complex Gaussians whose real and imaginary parts are independent. By construction the probability measure $\mu_{{\rm unif}, \mu}$ associated to the random field $F_{{\rm unif}, \mu}$ is supported on complex-valued fonctions $u$ such that $\Delta u = \mu^2 u$, in other words $F_{{\rm unif}, \mu} $ is almost surely a $\mu^2$-eigenfunction of the (geometric) Laplace operator $\Delta$. In particular $F_{{\rm unif}, \mu} $ is almost surely a smooth function on $\R^d$. 
  
\medskip

We shall now similarly define invariant monochromatic Gaussian random wave on symmetric spaces of non-compact types. We first set some notations.

\subsection{Notation}\label{s:notation}
Let $G$ be a non-compact real connected simple Lie group with associated symmetric space $X=G/K$. Fix a left invariant, bi-$K$-invariant metric on $G$.

Write $\g$ and $\kg$ for the Lie algebras of $G$ and $K$ and $\pg$ for the orthocomplement of $\kg$ with respect to the Killing form of $\g$ and $\ag$ for a maximal abelian subspace of $\pg$. Using a subscript $\C$ to denote 
complexifications, let $\mathcal{C} \subset i\ag^* \subset \ag_{\C}^*$ be the Weyl chamber corresponding
to a choice of positive roots for $(\g_\C , \ag_\C )$, and let $\rho$ be the corresponding half-sum
of positive roots. The direct sum of real root spaces for the chosen positive roots is a
Lie subalgebra, say $\nn$, of $\g$, and if $A$ and $N$ are the subgroups $\exp_G (\ag )$ and $\exp_G (\nn )$ of $G$, the map $(n, a, k) \mapsto nak$ is a diffeomorphism between $N \times A \times K$ and $G$ (Iwasawa decomposition). When the Iwasawa decomposition of an element $x \in G$ is $n \exp_G (H) k$ we write $H(x)=H$ for the $\ag$-component.

\medskip
\noindent
{\it Example.} The two main models of the hyperbolic plane are the Poincaré upper half-plane $\mathbf{H}$ and the unit disc $\D$. In the first case on can take $G = \SL_2 (\R)$ and in the latter $G =\mathrm{PSU}(1,1)$. Some computations are simpler in one model and some are simpler in the other one. We will switch between the two model leaving the reader and the context to decide whether $G = \SL_2 (\R)$ or $\mathrm{PSU}(1,1)$.  In both cases $K$ is isomorphic to $\SO_2$ and we can identify $\ag^*$ with $\R$ and $\mathcal{C}$ with $i \R_{>0}$. Using this identification $\rho \in \ag^*$ is equal to $\frac12 \in \R$. 

Let $g \in G=\SL_2 (\R)$. Writing 
$$g = \left( \begin{array}{cc} a & b \\ c & d \end{array} \right) \quad \mbox{and} \quad \frac{ai+b}{ci+d} = x + i y, $$ 
the Iwasawa decomposition of $g$ is
$$g = \left( \begin{array}{cc} 1 & x \\ 0 & 1 \end{array} \right) \left( \begin{array}{cc} y^{\frac12} & 0 \\ 0 & y^{-\frac12}  \end{array} \right) K$$
so that $H(g) = \frac12 \log y$.  

\medskip

\subsection{Spherical functions}
Suppose $s$ is in $\ag^*_\C$ and $b$ is in $K$. Define
$$e_{s , b} : G \to \C ; \ g \mapsto e^{-(s +\rho) (H (g^{-1}b ) )}.$$ 
Then $e_{s , b}$ defines a smooth function on $X$ that is an eigenfunction of the Laplace operator on $X$,
with eigenvalue $\|s\|^2 + \|\rho \|^2$. It plays the role of the exponentials $e^{i \mu \langle x-y , \xi \rangle}$ on $\R^d$. 

If $(s_1, b_1)$ and $(s_2, b_2)$ are elements of $\ag^*_\C \times K$, then $e_{s_1 , b_1}$ and $e_{s_2 , b_2}$ coincide if and
only if there is an element $w$ in the Weyl group of $(\g_\C, \ag_\C )$ such that $s_1 = w s_2$ and if $b_1$ 
and $b_2$ have the same image in the quotient $B = K/M$, where $M$ is the centralizer of $\ag$ in
$K$. Each of the $e_{s , b}$ thus coincides with exactly one of the $e_{s^+ , b}$'s, where $s^+$ runs through
the closure $\Lambda^+$ of $\mathcal{C}$ in $i\ag^*$.

\medskip
\noindent
{\it Example.} When $X$ is the hyperbolic plane then $B$ is the circle at infinity. In the disc model $\mathbf{D}$, for an eigenvalue $1/4 + r^2 \in \R$ associated to $s = ir$, we have 
$$e_{s , b} (z) = e^{\left( \frac12 + ir \right) \langle z , b \rangle} \quad (z \in \mathbf{D} , \ b \in B ),$$
where $\langle z , b \rangle$ is the signed distance to $0$ of the horocycle through the points $z$ and $b$, so that 
$e^{\langle z , b \rangle} = \frac{1-|z|^2}{|z-b|^2}$ is the Poisson kernel of the unit disc. In the upper half plane model, if $b$ is taken to be the point at infinity, we have $e_{s , b} (z) = y^{1/2 +ir}$. 

\medskip
We call \emph{spherical function} a bi-$K$-invariant function on $G$ (or in other words a function on $K \backslash X = K \backslash G / K$) that is also an eigenfunction of the Laplace operator on $X$. A theorem of Harish-Chandra \cite{HC} states that for each $s \in \Lambda^+$, 
$$\varphi_s : x \in G \mapsto \int_B e_{s , b} (x) db$$
is a spherical function. And every spherical function for the pair $(G,K)$ is of this form. In particular $\varphi_s$ is the only spherical function of eigenvalue $\|s\|^2 + \|\rho \|^2$ such that $\varphi_s(e) = 1$.

\begin{definition} \label{d:g1}
The {\it invariant (or isotropic) monochromatic Gaussian random wave with parameter $s \in \Lambda^+$ on $X$} is the standard complex Gaussian random field $F_{{\rm unif}, s} : X \to \C$ whose covariance kernel is 
$$\mathbb{E} [F_{{\rm unif}, s} (x) F_{{\rm unif}, s} (y)] = \varphi_s (x^{-1} y). $$
\end{definition}

\medskip
\noindent
{\it Example.} When $X$ is the hyperbolic plane and $s = ir$, then in the disc model we have 
$$\varphi_s (z) = \Phi_{r,0} (z),$$
where we more generally denote by $\Phi_{r,n}$ the family of generalized spherical functions defined in the disc model by 
$$e^{\left( \frac12 + ir \right) \langle z , b \rangle} = \sum_{n \in \Z} \Phi_{r,n} (z) b^n, \quad b \in B;$$ 
see \cite[Theorem 4.16]{Helgason}.

As it follows from the proof of \cite[Theorem 4.2]{Helgason} one could then alternatively describe $F_{{\rm unif}, s}$ as 
$$F_{{\rm unif}, \mu} (z) =  \sum_{n \in \Z} c_n \Phi_{r,n} (z)$$
where $(c_n )_{n \in \Z}$ are standard complex Gaussians whose real and imaginary parts are independent. 
 
\medskip

Note that it makes sense to study monochromatic Gaussian random waves in the discrete setting, as well, for instance for a regular tree, which is a symmetric space for its automorphism group. This direction has been initiated by Csoka, Gerencser, Harangi and Virag in \cite{CGHV} who used the notion to give new bounds on the independence ratio of random regular graphs. The deepest result in this direction for now is due to Backhausz and Szegedy \cite{BS}, who proved the Berry-type result that any (almost) eigenvector of a large random $d$-regular graph is BS-close to the monochromatic Gaussian eigenwave on the tree. 

We shall now provide more details on the construction of these random fields. 

\subsection{A Gaussian field on $B$}
The complex topological vector space $\mathcal{D} (B) = C^{\infty} (B )$ (equipped with the $C^\infty$-topology).  Let $\mathcal{D} ' (B )$ be the space of distributions, or continuous linear functionals $T$ on $\mathcal{D} (B )$. Since $\mathcal{D} (B)$ is a nuclear space, we may follow Hida \cite[\S 6.2]{Hida} and define the complex white noise in that context as a probability space $(\mathcal{D} ' (B ) , \mathcal{B} , \mu)$. Here $\mathcal{B}$ is the cylindrical $\sigma$-algebra on $\mathcal{D}' (B)$, i.e. the smallest $\sigma$-algebra such that 
for any $f \in \mathcal{D} (B)$, the function 
$$X_f : \mathcal{D} '(B) \to \C ; \ T \mapsto T(f)$$
is mesurable. Then $X_f$ defines a random variable on $(\mathcal{D} ' (B ) , \mathcal{B} , \mu)$ such that 
$$\mathbb{E} [ X_f ] = \int_{\mathcal{D} '(B)} T(f) d \mu (T) = 0 \quad \mbox{and} \quad \mathbb{E} [|X_f |^2 ] =  \int_B |f|^2 db;$$
see \cite[Proposition 6.7]{Hida}. 

From this we get that the covariance matrix for the collection $(X_f)_{f\in\mathcal D(B)}$ is given by 
\begin{equation}\label{e:whitecov}
\text{Cov}(X_f, X_g) = (f,g),
\end{equation}
with
\begin{equation}\label{scalar}
(f,g) =  \int  f(b) \overline{g(b)} \, db=  \iint \delta(b-b') f(b) \overline{g(b')} \, dbdb', 
\end{equation}
meaning that the covariance kernel is $\delta(b-b')$.

\subsection{Actions of $G$} 
Let $s \in \Lambda^+$. Denote by $\pi_{s}$ the compact picture of the induced spherical representation associated to $s$. It is the representation of $G$ in $L^2 (B)$ given by 
$$[ \pi_{s } (g) f ] (b ) = e^{( -s - \rho) H (g^{-1} b)} f (g^{-1} b) .$$
Since $s$ is imaginary, the representation $\pi_s$ is unitary. In other words it preserves the scalar product \eqref{scalar}; see e.g. \cite[Chap. VII, \S 2]{Knapp}. It extends to a representation of $G$ on $\mathcal{D} ' (B)$.

Since all the actions $\pi_s$ ($s \in \Lambda^+$) preserve the scalar product \eqref{scalar} they preserve the characteristic functional of the white noise. The measure $\mu$ is therefore $G$-invariant under all the $\pi_s$ ($s \in \Lambda^+$). It follows from the next proposition that $\mu$ is an \emph{ergodic} invariant measure. In fact much more is true:

\begin{prop} \label{L:Erg}
For any $s \in \Lambda^+$, the unitary representation $\Pi_s$ of $G$ in $L^2 (\mathcal{D}' (B) , \mu)$ induced by the action of $G$ on $\mathcal{D} '(B)$ by $\pi_s$ is mixing, i.e. for all $F_1 , F_2 \in L^2 (\mathcal{D}' (B) , \mu)$ we have
$$\int_{\mathcal{D} (B')} (\Pi_s (g) F_1 ) F_2 \ d \mu  \longrightarrow \left(  \int_{\mathcal{D}' (B)} F_1 \ d \mu \right) \left(  \int_{\mathcal{D} ' (B)} F_2 \ d \mu \right)$$
as $g$ tends to infinity in $G$. 
\end{prop}
\begin{proof} First note that the span of monomials 
\begin{equation} \label{monomials}
X_{f_1} \cdots X_{f_m} : \mathcal{D} ' (B) \to \C ; \ T \mapsto T( f_1) \cdots T (f_m) \quad (f_1 , \ldots , f_m \in \mathcal{D} (B))
\end{equation}
is dense in $L^2 ( \mathcal{D}'  (B) , \mu)$, see e.g. \cite[Corollary 1, p. 135]{Hida}. 

Now, for all $g \in G$ and $F \in L^2 ( \mathcal{D}'  (B) , \mu)$ we have
$$\Pi_s (g) F  = F \circ \pi_s (g^{-1}).$$
Since $\mu$ is $\pi_s (G)$-invariant, it follows that 
$$||\Pi_s (g) F ||_{L^2 ( \mathcal{D}'  (B) , \mu)}^2 = || F ||_{L^2 ( \mathcal{D}'  (B) , \mu)}^2.$$
So that approaching $F_1$ and $F_2$ by two linear combinations $F_1^{(n)}$ and $F_2^{(n)}$ of monomials \eqref{monomials}, we have
\begin{multline*}
\left| \int_{\mathcal{D}' (B)} (\Pi_s (g) F_1 ) F_2 d \mu - \int_{\mathcal{D}' (B)} (\Pi_s (g) F_1^{(n)} ) F_2^{(n)} d \mu \right| \\
\begin{split}
& \leq 
\left| \int_{\mathcal{D} ' (B)} (\Pi_s (g) ( F_1 -F_1^{(n)}) F_2 d \mu \right| + \left| \int_{\mathcal{D}' (B)} (\Pi_s (g) F_1^{(n)} ) (F_2 -F_2^{(n)} ) d \mu \right| \\
& \leq || F_1 -F_1^{(n)} ||_{L^2 ( \mathcal{D}'  (B) , \mu)} ||F_2||_{L^2 ( \mathcal{D}'  (B) , \mu)} \\
& \quad \quad + ||F_1^{(n)}||_{L^2 ( \mathcal{D}'  (B) , \mu)} ||F_2 -F_2^{(n)} ||_{L^2 ( \mathcal{D}'  (B) , \mu)}.
\end{split}
\end{multline*}
We are therefore reduced to proving Proposition \ref{L:Erg} when both $F_1$ and $F_2$ are monomials. 

Finally, for all $f_1 , \ldots , f_m, h_1 , \ldots , h_n$ in $\mathcal{D} (B)$ and $g \in G$, we evaluate
\begin{multline} \label{E:wick}
\int_{\mathcal{D} (B')} (\Pi_s (g) \cdot X_{f_1} \cdots X_{f_m} ) X_{h_1} \cdots X_{h_n} d \mu \\ 
= \mathbb{E} [ X_{\pi_s (g^{-1}) f_1 } \cdots X_{\pi_s (g^{-1}) f_m } X_{h_1} \cdots X_{h_n } ]
\end{multline}
using Isserlis' Theorem, a.k.a. Wick's probability Theorem. The latter indeed implies that the right hand side of \eqref{E:wick} is equal to 
\begin{multline*}
 \sum_{i=1}^m \sum_{j=1}^n \left( \mathbb{E} [X_{\pi_s (g^{-1}) f_i } X_{h_j}]  \right. \\ \left. \cdot \mathbb{E} [ X_{\pi_s (g^{-1}) f_1 } \cdots \widehat{X_{\pi_s (g^{-1}) f_i }} \cdots X_{\pi_s (g^{-1}) f_m } X_{h_1} \cdots \widehat{X_{h_j}} \cdots X_{h_n }] \right) \\ +\mathbb{E} [X_{\pi_s (g^{-1}) f_1 } \cdots X_{\pi_s (g^{-1}) f_m } ] \mathbb{E} [X_{h_1} \cdots X_{h_n }].
\end{multline*}
By Howe--Moore Theorem \cite[Theorem 2.2.20]{Zimmer} the matrix coefficients of $\pi_s$ vanish at infinity so that all the factors 
$$\mathbb{E} [X_{\pi_s (g^{-1}) f_i } X_{h_j}] = (\pi_s (g^{-1}) f_i , h_j )$$
tend to $0$ as $g$ tends to infinity. We conclude that, as $g$ tends to infinity, the integral
$$\int_{\mathcal{D} (B')} T(\pi_s (g^{-1}) f_1 ) \cdots T(\pi_s (g^{-1}) f_m ) T(h_1 ) \cdots T (h_n ) d \mu (T) $$
tends to 
\begin{equation*}
\mathbb{E} [X_{\pi_s (g^{-1}) f_1 } \cdots X_{\pi_s (g^{-1}) f_m } ] \mathbb{E} [X_{h_1} \cdots X_{h_n }]  = \mathbb{E} [X_{ f_1 } \cdots X_{f_m } ] \mathbb{E} [X_{h_1} \cdots X_{h_n }] 
\end{equation*} 
that is equal to
$$\left( \int_{\mathcal{D}' (B)} T( f_1 ) \cdots T(f_m ) d \mu (T) \right) \left( \int_{\mathcal{D} '(B)} T(h_1 ) \cdots T (h_n ) d \mu (T) \right).$$
This proves Proposition \ref{L:Erg} for monomials.
\end{proof}

\subsection{Gaussian random waves}
Let $\mathcal{E} (X)$ be the space of smooth functions on $X$ equipped with smooth topology. 
 
Given $s \in \Lambda^+$ we denote by $\mathcal{E}_s$ the closed subspace of $\mathcal{E} (X)$ generated by the translates of $\varphi_s$ under the left regular $G$-action. 

The map 
$$p_s :  \mathcal{D} '(B) \to \mathcal{E}_{s}; \ T \mapsto \int_B e_{s , b} dT(b)$$  
is $G$-equivariant with respect to the $G$-action on $\mathcal{D}' (B)$ induced by $\pi_{s}$ and the (left-regular) $G$-action on $\mathcal{E}_{s}$. 

\begin{definition}\label{d:gauss}
Pushing forward the probability space $(\mathcal{D} ' (B) , \mathcal{B} , \mu)$ by $p_s$ yields a probability space 
$$( \mathcal{E}_{s} , (p_s)_* (\mathcal{B} ) , \mu_{{\rm Gauss}, s} ),$$
the {\it Gaussian random wave} associated to $s \in \Lambda^+$ on the symmetric space $X=G/K$. 
\end{definition}

Gaussian random waves are probability measures invariant under the isometry group of the symmetric space. A similar construction has been considered by Afgoustidis \cite{Afgoustidis}. Note that it follows from Proposition \ref{L:Erg} that the $G$-invariant measure $\mu_{{\rm Gauss} , s}$ is ergodic under the $G$-action. 

We conclude this section by checking that Definitions \ref{d:gauss} and \ref{d:g1} agree: fix a function $f$ in the space $\mathcal D(X)$ of smooth compactly supported functions with the usual Schwartz topology (see Helgason \cite[p. 239]{Helgason}). By construction, the function
$$X^s_f : \mathcal E_s \to \C, \ g \mapsto (f,g)_{L^2 (X)} = \int_X f (x) \overline{g (x)} dx.$$ 
is $(p_s)_* (\mathcal{B} )$-mesurable and therefore defines a (complex) random variable on 
$$( \mathcal{E}_{s} , (p_s)_* (\mathcal{B} ) , \mu_{{\rm Gauss}, s} ).$$

\begin{prop}\label{l:covkernel}
The random variables $X^s_f$ have a centered (i.e. mean $0$) Gaussian distribution. The covariance is given by
$${\rm Cov}(X^s_f, X^s_g) = \iint \varphi_s(x^{-1}y) f(x) \overline{g(y)} \, dx dy.$$
The covariance kernel of the Gaussian field is therefore $\varphi_s(x^{-1}y)$, where $\varphi_s$ is the spherical function.
\end{prop}
\begin{proof}
The fact that $X^s_f$ is complex Gaussian follows from the fact that the process associated to the white noise $\mu$ is Gaussian. We can also check directly that both the real and imaginary parts of $X^s_f$ are real Gaussian. To do so we would to re-do the above construction with real functions (replacing in particular $e_{s,b}$ by its real and imaginary parts) and compute:
\begin{align*}
\mathbb{E} \left[ \exp (i t X^s_f) \right] &= \int_{\mathcal E_s} \exp(i t (f,g)_{L^2 (X)}) \, d\mu_{{\rm Gauss}, s} \\
&= \int_{\mathcal D'(B)} \exp \left( i t \int_X f(x) \int_B  e_{s,b} (x) \,dT(b)  \, dx \right) d\mu(T)\\
&= \int_{\mathcal D'(B)} \exp \left[ i   \int_B \left( t \int_X f(x) e_{s,b} (x)  \, dx \right) \,dT(b)\right] d\mu(T)\\
&= \exp \left[   -\frac{t^2}{2} \int_B \left(  \int_X f(x) e_{s,b} (x)  \, dx \right)^2 \,db\right],
\end{align*}
since 
$$f \mapsto C(f ) = \exp \left[ - \frac12 \int_B f^2 db \right]$$
is the characteristic functional of the white noise on $C^{\infty} (B , \R)$.

To compute the covariance kernel we write
\begin{align*}
 {\rm Cov}(X^s_f, X^s_g) &=  \mathbb{E}(X^s_f \overline{X^s_g}) \\
 &= \int (f, u)_{L^2 (X)} \overline{(g,u)_{L^2 (X)}} \, d\mu_{{\rm Gauss}, s} (u) \\
 &= \int\limits_{\mathcal D'(B)} \int\limits_X \overline{p_s(T)(x)} f(x) \, dx \int\limits_X  p_s(T) (y) \overline{g(y)} \, dy \, d\mu(T)\\
 &= \iint_{X\times X} \left( \int_{\mathcal D'(B)} \overline{p_s(T)(x)} p_s(T)(y) d\mu(T) \right) f(x) \overline{g(y)} \, dxdy ,
\end{align*}
and, by \eqref{e:whitecov}, the covariance kernel is given by
$$\int_B \overline{e_{s,b}(x)} e_{s,b}(y) \, db.$$
On the other hand we have:
$$\varphi_s (x^{-1} y) = \int_K e^{-(s + \rho) H(y^{-1} x k)} dk.$$
To conclude the proof we shall use the following identity. Let $g$ and $h$ be two elements in $G$ and let $k$ be an element in $K$. Writing $gk = u(gk) H(gk) n \in KAN$ we obtain, since $A$ normalizes $N$, 
\begin{equation} \label{eqH}
H(hgk) = H(gk) + H (h u(gk)).
\end{equation} 
It follows that 
$$\varphi_s (x^{-1} y) =  \int_K e^{-(s + \rho) H(xk)} e^{-(s + \rho) H(y^{-1} u(xk))} dk.$$
But, by \eqref{eqH}, we have $H(xk) = - H(x^{-1} u(xk))$ and therefore
$$\varphi_s (x^{-1} y) =  \int_K e^{(s + \rho) H(x^{-1} u(xk))} e^{-(s + \rho) H(y^{-1} u(xk))} dk.$$
Now the mapping $F_x : k \mapsto u(xk)$ is a diffeomorphism of $K$ and $(F_x)^* (dk) = e^{-2\rho H(x^{-1} k)}$. It follows that
\begin{align*}
\varphi_s (x^{-1} y) & =  \int_K e^{(s + \rho) H(x^{-1}k)} e^{-(s + \rho) H(y^{-1}k)} e^{-2\rho H(x^{-1} k)}dk \\
& = \int_K e^{(s - \rho) H(x^{-1}k)} e^{-(s + \rho) H(y^{-1}k)} dk \\
& = \int_B \overline{e_{s,b}(x)} e_{s,b}(y) \, db.
\end{align*}
Here we have used that $s \in i \mathfrak{a}^*$ so that $\overline{s} =- s$.
\end{proof}

\section{Random waves: Berry type conjectures} \label{S:berry}

\subsection{Berry's conjecture in BS form} Let $M$ be a $d$-dimensional closed Riemannian manifold. Recall from the introduction that we denote by $M_r$ the rescaling of $M$ by $r$, that is, we multiply every distance by $r$. A Riemannian metric is infinitesimally Euclidean, it follows from the definitions that, as $r\rightarrow \infty$, the sequence $(M_r)$ BS-converges toward the Dirac measure at $(\R^{d} , 0) \in \mathcal{M}^d$. Note that $\R^d$ being homogeneous, the limit measure does not depend on a particular choice of base point. We loosely say that $(M_r)$ BS-converges toward $\R^d$. 

Now, if $\phi:M\rightarrow \R$ is an eigenfunction of the
Laplace operator on $M$, then the very same function $\phi : M_r \rightarrow \R$ is an eigenfunction of the Laplacian on $M_r$
with eigenvalue $\lambda^{\prime}=\lambda / r^{2}$.  We conclude that if $\phi_n$ is a sequence of eigenvectors for the Laplace operator, with eigenvalues $\lambda_n = \alpha_n^2$, then any weak limit of $\mu_{M_{\alpha_n} , \phi_n }$ is supported in
$$\{ [\R^d , p , \psi ] \in \mathcal{E}^d \; : \; p \in \R^d , \ \Delta \psi = \psi \}.$$ 
We therefore loosely identify such a weak limit with a random field on $\R^d$. We may now recall our version --- Conjecture \ref{BerryConj} --- of Berry's conjecture.

\medskip
\noindent
{\bf Conjecture 1.} {\it Let $M$ be a compact, negatively curved
manifold. Let $(\phi_{n} )$ be an orthonormal basis of $L^2 (M)$ that consists of eigenvectors for the Laplace operator, with eigenvalues $\lambda_n = \alpha_n^2$. Then $(M_{\alpha_n}, \phi_{n})$ BS converges to $(\R^d,F_{{\rm unif} , 1})$ --- the isotropic monochromatic Gaussian random Euclidean wave with eigenvalue $1$.}

\subsection{Relation with QUE} \label{S:relQUE}

Let $M$ be a compact $d$ dimensional manifold with normalized volume form $d\mathrm{vol}_M$. Let $\phi_{n}$ be a sequence of eigenfunctions of eigenvalue $\lambda_{n} = \alpha_n^2$ of
$L^{2}$-norm $1$. Denote by $\nu_{n}$ the probability measure on $M$ defined by the density function
$\phi_{n}^{2}$. Let $X_{n}$ be the value of $\phi_{n}^{2}$ at a $d\mathrm{vol}_M
$-random point of $M$. Then $X_{n}$ is a bounded random variable with distribution
$\rho_{n}$ that is equal to the probability measure on $\R^{+}$ obtained as the pushforward of $\mu_{M_{\alpha_n} , \phi_n }$ by the map 
$$\mathcal{E}^d \to \R^+; \quad (M , p , \phi) \mapsto \phi (p)^2,$$  .

Let $M_{n}=M_{\alpha_{n}}$ be the rescaled manifold. Assume that a weak form of Conjecture 1 holds, namely that the sequence $(M_{n},\phi_{n})$ BS converges to $(\R^{d},F)$ where $F$ is an invariant random 
(not necessarily Gaussian) Euclidean eigenwave with eigenvalue $1$. Also assume that $\nu_{n}$ weakly converges to some probability measure $\nu$ on $M$. Since the $\phi_{n}$ are smooth, each measure $\nu_{n}$ is 
absolutely continuous with respect to $d\mathrm{vol}_M$. Let 
$$\nu=\nu_{c}+\nu_{s}$$ 
be the continuous-singular decomposition of $\nu$ with respect to $d\mathrm{vol}_M$ (Lebesgue's decomposition Theorem).

By BS convergence, the sequence of random variables $X_{n}$ weakly converges to the random variable $F(0)^2$. Note that since $F$ is a random function the expression $F(0)^2$ is a random variable in $\R^+$. The distribution $\rho$ of the random variable $F(0)^2$ is the weak* limit of $\rho_{n}$,\footnote{In this paragraph we use $\rho$ to denote a distribution; this has no relation with the $\rho$ of the previous section that will not appear here.} and we have
\[
\int xd\rho_{n}=\mathbb{E}X_{n}=1\text{. }%
\]
Let the energy of $F$ be
\[
e(F)=\mathbb{E}F^{2}(0)=\int xd\rho\text{.}%
\]
Since $\rho_{n}$ weak* converges to $\rho$, we have $e(F)\leq1$. A convenient
way to express the possible deficit is as follows. Let 
$$\tau_{n}=(\phi_{n}^2)_* \nu_{n}$$ 
be the push-forward measure of $\nu_{n}$ by $\phi_{n}^{2}$,
that is, let
\[
\tau_{n}(A)=\nu_{n}((\phi_{n}^{2})^{-1}(A))\text{ \ (}A\subseteq\R^{+}\text{ Borel).}%
\]
Then
\begin{equation} \label{A}
d\tau_{n}(x)=xd\rho_{n}(x)\text{ (}x\in\R^{+}\text{).} 
\end{equation}
Let us compactify $\R^{+}$ by adding $\infty$, call this space
$[0,\infty]$. Then $\tau_{n}$ is a probability measure on $[0,\infty]$ and by \eqref{A} the sequence $(\tau_{n})$ weakly
converges to the probability measure $\tau$ on $[0,\infty]$, defined by 
\[
d\tau(x)=xd\rho(x)\text{ }(x\in\R^{+})\text{.}
\]
The deficit $1-e(F)$ will be equal to the `amount of mass traveling to
$\infty$', that is, we have
\begin{equation} \label{B}
1-e(F)=\tau(\{\infty\})=\lim_{K\rightarrow\infty}\liminf_{n\rightarrow\infty
}\tau_{n}([K,\infty]) . 
\end{equation}
It would be desirable to write $\lim$ instead of $\liminf$ above but we do
not control how weak convergence of $\tau_{n}$ evolves at the point $K$.

Before proving (the stronger) Theorem \ref{weakQUE} we note that the singular part of $\nu$ has mass at most the loss of
energy in the limit:

\begin{prop} \label{P15}
We have $\nu_{s}(M)\leq1-e(F)$. In particular, if $e(F)=1$, then $\nu$ is
absolutely continuous wrt volume.
\end{prop}
\begin{proof} Let $S$ be the support of the singular part $\nu_{s}$. By definition of the continuous-singular decomposition, the measures
$\nu_s$ and $d\vol_M$ are singular, i.e. there exist two disjoint Borel subsets $A_1$ and $A_2$ in $M$ whose union is $M$ and such that $\nu_s$ is zero on all measurable subsets of $A_2$ while $d\vol_M$ is zero on all mesurable subsets of $A_1$. In particular the Lebesgue measure of $S$ is zero and, given any positive $\varepsilon$, there exists $r>0$ such that the open
$r$-neighborhood $O$ around $S$ has volume less than $\varepsilon$. Weak
convergence of $\nu_{n}$ implies
\begin{equation} \label{C}
\liminf\nu_{n}(O)\geq\nu(O)\geq\nu_{s}(S)=\nu_{s}(M)\text{.} 
\end{equation}

For a fixed $n$, let $K_n=\nu_{n}(O)$ and let
\[
B=(\phi_{n}^{2})^{-1}([\frac{K_n}{\sqrt{\varepsilon}},\infty])=\left\{  x\in O\mid
\phi_n^{2}(x)\geq\frac{K_n}{\sqrt{\varepsilon}}\right\}  \text{.}%
\]
Then
\[
\nu_{n}(O\setminus B)=\int\limits_{O\setminus B}\phi_{n}^{2}(x)d\mathrm{vol}%
\leq\mathrm{vol}(O\setminus B)\frac{K_n}{\sqrt{\varepsilon}}\leq K_n \sqrt
{\varepsilon}%
\]
which implies
\[
\nu_{n}(B)\geq K_n (1-\sqrt{\varepsilon})\text{. }%
\]
Summarizing, for all $n>0$ we have
\[
\tau_{n}([\frac{\nu_{n}(O)}{\sqrt{\varepsilon}},\infty])\geq\nu_{n}%
(O)(1-\sqrt{\varepsilon})\text{.}%
\]
In particular, using \eqref{C}, we have
\[
\lim\inf_{n\rightarrow\infty}\tau_{n}([\frac{\nu_{n}(O)}{\sqrt{\varepsilon}%
},\infty])\geq\lim\inf_{n\rightarrow\infty}\nu_{n}(O)(1-\sqrt{\varepsilon
})\geq(1-\sqrt{\varepsilon})\nu_{s}(M)\text{.}%
\]
Choosing $\varepsilon>0$ to be arbitrarily small and using \eqref{B}, this proves
the Proposition. 
\end{proof}

\begin{proof}[Proof of Theorem \ref{weakQUE} from the Introduction]
Up to passing to a subsequence we may suppose that the sequence $d\nu_n := \phi_n^2 \, d\vol_M$ is weakly convergent; let $\nu$ be its limit.
To prove that $\nu$ coincides with the volume measure $\vol_M$, it is therefore enough to prove that for any continuity set $B \subset M$  (i.e. such that $\vol_{M} (\partial B) = 0$) of positive measure, we have:
\begin{equation} \label{w*conv}
\lim_{n \to \infty} \nu_n (B) = \vol_M (B).
\end{equation}

Let us fix a continuity set $B \subset M$ with $\vol_M (B) >0$ and prove \eqref{w*conv}. Let $\mu_{M_{n} , \phi_n }^B$ be the probability measure on $\mathcal{E}^d$ obtained by pushing forward the probability measure 
$$U \mapsto \frac{\vol_{M_{n}} (U \cap B)}{\vol_{M_{n}} (B )} =  \frac{\vol_M (U \cap B)}{\vol_M (B )}$$
on $M_n$, under the map \eqref{map:pf}. In other words $\mu_{M_{n} , \phi_n }^B$ is obtained by sampling only into $B$. 

\begin{lem} \label{L:16a}
The set $\{ \mu_{M_{n} , \phi_n }^{B} \}$ of probability measures on $\mathcal{E}^d$ is relatively compact with respect to the topology of weak convergence. 
\end{lem}
\begin{proof} Note that if $A$ is a measurable subset of $\mathcal{E}^d$ we have:
$$\mu_{M_{n} , \phi_n } (A) \geq c \mu_{M_{n} , \phi_n }^{B} (A),$$
where $c$ is a positive constant independent of $n$ --- one can take $c = \vol_M (B)$. Since by assumption the sequence of probability measures $(\mu_{M_{n} , \phi_n })$ weakly converges to a probability measure on $\mathcal{E}^d$, the collection of measures $\{  \mu_{M_{n} , \phi_n }^{B} \}$ has to be tight, and the lemma follows from Prokhorov's theorem.
\end{proof}

\begin{lem} \label{L:16b}
Any  weak limit of a converging subsequence of $(\mu_{M_{n} , \phi_n }^{B} )$ is supported in 
$$\{ [\R^d , p , \psi ] \in \mathcal{E}^d \; : \; p \in \R^d , \ \Delta \psi = \psi \}$$ 
and is translation-invariant.
\end{lem}
\begin{proof} The first part of the lemma follows from the construction since, denoting by $\Delta_n$ the Laplace operator on $M_n$, we have $\Delta_n \phi_n = \phi_n$. It therefore remains to prove that any weak limit is translation-invariant. 

Recall that each measure $\mu_{M_{n} , \phi_n }$ has a natural lift $\tilde{\mu}_{M_{n} , \phi_n }$ on the space $T^1 \mathcal{E}^d$ of isometry classes of rooted unit tangent bundles colored by a function. Since by assumption Conjecture \ref{BerryConj} holds, the sequence $(\tilde{\mu}_{M_{n} , \phi_n })$ weakly converges toward the natural lift of the measure associated with $(\R^d , F_{{\rm unif}, 1})$ that is supported in 
$$\{ [\R^d , p , \psi ] \in \mathcal{E}^d \; : \; p \in \R^d , \ \Delta \psi = \psi \}.$$ 
The flow \eqref{geodflow} naturally extends to a flow 
$$\mathbf{g}_t : T^1 \mathcal{E}^d \to T^1 \mathcal{E}^d$$
and the limit measure associated with $(\R^d , F_{{\rm unif}, 1})$ is invariant under the flow $\mathbf{g}_t$ (by homogeneity of $\R^d$).

Now consider the lifts $\tilde{\mu}^B_{M_{n} , \phi_n }$ of the measures $\mu_{M_{n} , \phi_n }^{B}$. Passing to a subsequence, we may suppose that the sequence $(\mu_{M_{n} , \phi_n }^{B} )$ weakly converges. Given a subset 
$$\mathcal{U} = ( \Omega , U_{[N , p , \phi]} ) \subset T^1 \mathcal{E}^d,$$
where $\Omega \subset \mathcal{E}^d$ is an open subset and $U_{[N , p , \phi ]}$ is an open subset of the unit sphere $T_p^1 N$, the measure  $\tilde{\mu}^B_{M_{n} , \phi_n }$ is defined by 
\begin{equation*}
\begin{split}
\tilde{\mu}^B_{M_{n} , \phi_n } (\mathcal{U} ) & = \int_\Omega \int_{U_{[N , p , f]}} dS_{T_p^1N}^{d-1} d \mu_{M_{n} , \phi_n }^{B} ([N , p , \phi]) \\
& = \int_{\omega \cap B} \int_{U_p} d S_{T_p^1 M_n}^{d-1} \frac{d \vol_{M_n} (p)}{\vol_{M_n} (B)} \\
& = \frac{\omega_{T^1 M_n} (U \cap \tilde{B})}{\omega_{T^1 M_n} (\tilde{B})}.
\end{split}
\end{equation*}
Here $dS_{T_p^1 N}^{d-1}$ is the usual Lebesgue measure on the unit sphere $T_p^1 N$, the open subset 
$$U = (\omega , U_p)  \subset T^1M_n$$ 
is the preimage of $\mathcal{U}$ by the map
$$T^1 M_n \to T^1\mathcal{E}^d; \quad (p,v) \mapsto [M_n , p , v,  \phi_n ],$$
we denote by $\tilde{B}$ the preimage of $B$ in $T^1 M_{n}$ and write $\omega_{T^1 M_n}$ for the Liouville measure of $M_n$. 

We shall now prove that
\begin{equation}\label{e:transinv}
\frac{\omega_{T^1M_{n}} (\mathbf{g}_{T}^{-1}( U) \cap \tilde{B}) - \omega_{T^1 M_{n}} (U \cap \tilde{B})}{\omega_{T^1 M_{n}} ( \tilde{B} )} \to 0
\end{equation}
as $n\to +\infty$. It follows that the lifted measure $\tilde{\mu}^B_{M_{n} , \phi_n }$ and $(\mathbf{g}_T)_* \tilde{\mu}^B_{M_{n} , \phi_n }$ are asymptotically equal and both converge to the lift to $T^1 \mathcal{E}^d$ of the weak limit of $(\mu_{M_{n} , \phi_n }^{B} )$. Pushing forward the measures $\tilde{\mu}^B_{M_{n} , \phi_n }$ and $(\mathbf{g}_T)_* \tilde{\mu}^B_{M_{n} , \phi_n }$ on $T^1 \mathcal{E}^d$ to measures on $\mathcal{E}^d$ we get Lemma \ref{L:16b}.

It remains to prove \eqref{e:transinv}. For any fixed positive real number $R$, denote by $N_R^{n}(\tilde{B})$ the $R-$neighborhood of $\tilde{B}$ on $T^1M_{n}$.
Notice first that since $M_n$ is obtained from $M=M_0$ by rescaling the metric by $\alpha_n$, and since we assumed $\vol_{M} (\partial B) = 0$, we have:
\begin{equation}\label{e:folner}
\frac{\omega_{T^1M_{n}} (N_R^n (\tilde{B}))}{\omega_{T^1M_{n}}(\tilde{B})} = \frac{\omega_{T^1M} (N_{R/\alpha_n}^0 (\tilde{B}))}{\omega_{T^1M}(\tilde{B})} \to 1
\end{equation}
as $n$ tends to infinity. Using the invariance of the Liouville measure under $\mathbf{g}_T$, we can then rewrite
\begin{align*}
&\omega_{T^1M_{n}} (\mathbf{g}_T^{-1} ( U) \cap  \tilde{B}) - \omega_{T^1M_{n}} (U \cap \tilde{B})  \\
&\qquad = \omega_{T^1M_{n}} ( U \cap \mathbf{g}_T \tilde{B}) - \omega_{T^1M_{n}} (U \cap \tilde{B})\\
&\qquad = \omega_{T^1M_{n}}(U \cap (\mathbf{g}_{T} \tilde{B} \setminus \tilde{B})) - \omega_{T^1M_{n}}(U \cap (\tilde{B}\setminus \mathbf{g}_{T} \tilde{B})).
\end{align*}
Since 
$$\mathbf{g}_{T} \tilde{B} \subset N_{T}^n(\tilde{B}) \quad \mbox{and} \quad \tilde{B} \subset N_T^n (\mathbf{g}_{T} \tilde{B}),$$
we have:
$$\frac{\omega_{T^1M_{n}}(U \cap (\mathbf{g}_{T} \tilde{B} \setminus \tilde{B}))}{\omega_{T^1M_{n}}(\tilde{B})}\leq \frac{ \omega_{T^1M_{n}}(N_{T}^n(\tilde{B}) \setminus \tilde{B})}{\omega_{T^1M_{n}}(\tilde{B})}$$
and 
\begin{equation*}
\begin{split}
\frac{\omega_{T^1M_{n}}(U \cap (\tilde{B}\setminus \mathbf{g}_{T} \tilde{B}))}{\omega_{T^1M_{n}}(\tilde{B})} & \leq \frac{ \omega_{T^1M_{n}}( N_T^n (\mathbf{g}_{T} \tilde{B}) \setminus \mathbf{g}_{T} \tilde{B}))}{\omega_{T^1M_{n}}(\tilde{B})} \\
& \leq \frac{\omega_{T^1M_{n}}( N_T^n (\mathbf{g}_{T} \tilde{B}) \setminus \mathbf{g}_{T} \tilde{B}))}{\omega_{T^1M_{n}}(\mathbf{g}_T \tilde{B})} 
\end{split}
\end{equation*}
and it follows from \eqref{e:folner} that they both tend to $0$ as $n\to+\infty$. This proves \eqref{e:transinv} and concludes the proof of Lemma \ref{L:16b}.

\end{proof}

Using the obvious fact that:
\begin{equation} \label{E:convcomb}
\mu_{M_{n} , \phi_n } = \vol_M (B) \mu_{M_{n} , \phi_n }^{B} + \vol_M ({}^c B) \mu_{M_{n} , \phi_n }^{{}^c B},
\end{equation}
where ${}^c B$ is the complement of $B$ in $M$, we finally prove:

\begin{lem} \label{L:16}
The sequence $(\mu_{M_{n} , \phi_n }^{B} )_n$ weakly converges toward the isotropic monochromatic Gaussian random Euclidean wave with eigenvalue $1$.
\end{lem}
\begin{proof} By Lemma \ref{L:16a}, it is enough to prove that $F_{{\rm unit}, 1}$ is the only possible weak limit of $(\mu_{M_{n} , \phi_n }^{B} )_n$. So let $\mu_1$ be a weak limit of a converging subsequence $(\mu_{M_{n_j} , \phi_{n_j} }^{B})$. 

Working similarly with the complement ${}^c B$ of $B$ in $M$, we may, and will, suppose that both $\mu_{M_{n_j} , \phi_{n_j} }^B$ and $\mu_{M_{n_j} , \phi_{n_j} }^{{}^c B}$ weakly converge. Denote by $\mu_1$ and $\mu_2$ their respective limits. By Lemma \ref{L:16b} these measures are both translation-invariant, and it follows from \eqref{E:convcomb} that --- as the weak limit of $\mu_{M_{n} , \phi_n }$ --- the process $F_{{\rm unif}, 1}$ can be decomposed as a convex sum
$$\vol_M (B) \mu_1 + \vol_M ({}^c B) \mu_2$$
of two translation-invariant measures.  But being ergodic the process $F_{{\rm unif}, 1}$ cannot be decomposed as a non-trivial convex sum of translation-invariant measures. It then follows that both $\mu_1$ and $\mu_2$ are equal to $F_{{\rm unif}, 1}$. 
\end{proof}

To conclude the proof of Theorem \ref{weakQUE}, note that --- as in the proof of Proposition \ref{P15} --- we have:
$$\nu_n (B) = \vol_M (B) \mathbb{E} (X_n^B ) = \vol_M (B) \int x d\rho_n^B (x)$$
where $X_n^B$ is the value of $\phi_n^2$ at a random point of $B$ with respect to the probability measure $d\vol_M / \vol_M (B)$, and we denote by $\rho_n^B$ the distribution of $X_n^B$; it is equal to the probability measure on $\R^+$ obtained by pushing forward $\mu_{M_{n} , \phi_n }^B$ by the continuous map 
$$\mathcal{E}^d \to \R^+; \ [M,p,\psi] \mapsto \psi (p)^2.$$
Now, by Lemma \ref{L:16}, the sequence $(\mu_{M_{n} , \phi_n }^B)$ weakly converges toward the isotropic monochromatic Gaussian random Euclidean wave $F=F_{{\rm unif},1}$. It follows that the sequence $(\rho_n^B )$ weakly converges toward the distribution
$\rho$ of $F(0)^2$, and since the energy $e(F)$ is equal to $1$, the sequence of measures $\tau_n^B$ defined by 
$$d\tau_n^B (x) = x d\rho_n^B (x) \quad ( x \in \R^+),$$
weakly converges toward the probability measure $\tau$ on $\R^+$ defined by 
$$d\tau (x) = x d\rho (x) \quad (x \in \R^+).$$
We conclude that 
$$\int x d\rho_n^B (x) \to 1$$
and therefore 
$$\nu_n (B) \to \vol_M (B)$$
as $n$ tends to infinity.
\end{proof}

\begin{unremark} The proof of Theorem \ref{weakQUE} only uses that the limiting wave $F$ is ergodic and has energy $e(F) =1$.
\end{unremark}

\subsection{Level aspect} \label{S:level} As explained in the Introduction the BS formulation of Berry's conjecture immediately suggest a similar conjecture but regarding the level aspect --- Conjecture \ref{BC2} of the Introduction. Keeping  notations as in the preceding paragraphs, we first revisit in more details this conjecture before raising more questions. 

Recall that to any uniform, torsion free, discrete subgroup of $G$ and to any $\Gamma$-invariant function $\phi \in \mathcal{E}_s$ of normalized $L^2$-norm $1$ on $\Gamma \backslash G$ we have associated a $G$-invariant probability measure $\mu_{\phi}$ on $\mathcal{E}_s$. 

A family of lattices in $G$ is \emph{uniformly discrete} if there is an identity neighborhood in $G$ that intersects trivially all of their conjugates. For torsion
free lattices this is equivalent to saying that there is a uniform lower bound for the injectivity radius of the corresponding $X$-manifolds.
 
Let $\Gamma_n $ be a uniformly discrete sequence of lattices in $G$ that BS-converges toward the trivial IRS. Then there exists a sequence $R_n \to \infty $ such that 
$$\alpha_n = \frac{\mathrm{vol} (\Gamma_n \backslash G )_{<R_n}}{\mathrm{vol} (\Gamma_n \backslash G )} \to 0.$$
Note that necessarily $R_n = O(\log \vol(\Gamma_n \backslash G))$, otherwise $(\Gamma_n \backslash G )_{<R_n} = \Gamma_n \backslash G$.
\begin{unremark} It follows from \cite[\S 5]{7samurai} that, for congruence groups, we can take $R_n = c \log \mathrm{vol} (\Gamma_n \backslash G )$ so that $\alpha_n \leq \mathrm{vol} (\Gamma_n \backslash G )^{\beta}$ with $\beta$ positive. 
\end{unremark}

For later purposes let us fix $r_n \leq c' R_n$ with $0<c'<1$ a sequence that tends to infinity with 
\begin{equation} \label{limBS}
r_n \mathrm{vol} (B_G (e , r_n) ) \alpha_n \to 0.
\end{equation}

Let $s_0 \in \mathcal{C}$ (recall the notation in Section \ref{s:notation}) and let $\delta$ be a positive real number such that the $\delta$-neighborhood\footnote{Here we equip $\mathfrak{a}^*$ with the metric induced by the Killing form.} $I_\delta (s_0)$ of $s_0$ in $i \mathfrak{a}^*$ is contained in $\Lambda^+$. We denote by $N(\delta , \Gamma_n)$ the dimension of the subspace $\mathcal{H}_{\delta}^{(n)} \subset L^2 (\Gamma_n \backslash X)$ of smooth $\Gamma_n$-invariant functions on $X$ that is spanned by a maximal orthogonal family of functions $\phi \in L^2 (\Gamma_n \backslash X)$ that satisfy
$$\int_{\Gamma_n \backslash G} |\phi |^2 =1 \mbox{ and } \Delta \phi = (\|\rho\|^2 + \|s\|^2) \phi $$
with $s$ in $I_\delta (s_0)$. Note that if we write  $\lambda_0 = \|\rho\|^2 + \|s_0\|^2$, we have
$$N(\delta, \Gamma_n) = \# \{ i \; : \; \lambda_i^{(n)} \in [\lambda_0 -\delta' , \lambda_0 + \delta'] \}, $$
for some $\delta'$ depending on $\delta$ and $s_0$, where $(\lambda_i^{(n)})_{i\in\N}$ is the sequence of eigenvalues of the Laplacian on $\Gamma_n \backslash X$ with each eigenvalue appearing a number of time equal to its multiplicity.

\medskip
\noindent
{\it Example.} For $G = \SL_2 (\R)$, recall that we can identify $\ag$ with $\R$ and $\mathcal{C}$ with $i \R_{>0}$. Using this identification $\rho \in \ag^*$ is equal to $\frac12 \in \R$.  In this case, $I_\delta(s_0)$ is an interval in $[0,+\infty)$ corresponding to an interval $[\lambda_0 -\delta' , \lambda_0 + \delta'] \subset [\frac14, +\infty)$ in the spectrum, with $\lambda_0 = \frac14 + |s_0|^2$.

\medskip 

\begin{lem} \label{L:deltan}
There exists a sequence $(\delta_n)$ of  positive real numbers that converges to $0$ and satisfies the two following properties.
\begin{enumerate}
\item As $n$ tends to infinity, we have 
$$\frac{N(\delta_n ,\Gamma_n)}{\mathrm{vol} (\Gamma_n \backslash G)} \sim \mu_{\rm Planch} (I_{\delta_n} (s_0 ) ) .$$
\item For all $n \geq 1$, we have $\delta_n \geq r_n^{-1}$.
\end{enumerate}
\end{lem}
\begin{proof} By hypothesis $s_0$ correspond to an eigenvalue in the interior of the $L^2$-spectrum of $X$. Now  \cite[Theorem 1.2]{7samurai} implies that the spectral measure of $\Gamma_n \backslash X$ weakly converges toward the spectral measure of $X$. For all positive $\delta$ such that 
$I_\delta (s_0)$ is contained in the tempered spectrum, we have
$$\frac{N(\delta,\Gamma_n)}{\mathrm{vol} (\Gamma_n \backslash G)} \to  \mu_{\rm Planch} (I_\delta (s_0 ) ) $$
as $n$ tends to infinity. It follows the sequence 
$$\delta_n = \mathrm{inf} \left\{ \delta \in [r_n^{-1} , +\infty) \; : \; (1-\delta ) \mu_{\rm Planch} (I_\delta (s_0 ) ) \leq \frac{N(\delta,\Gamma_n)}{\mathrm{vol} (\Gamma_n \backslash G)} \leq (1+\delta ) \mu_{\rm Planch} (I_\delta (s_0 ) ) \right\},$$
satisfies the desired properties.  
\end{proof}

For each $n$, fix an orthonormal basis $(\phi_j^{(n)})$ of $\mathcal{H}_{\delta_n}^{(n)}$. 
The following provocative conjecture is equivalent to Conjecture 3 of the Introduction. 

\begin{conjecture} \label{C1}
Suppose that the representation of $G$ in $\overline{\bigoplus_n L_0^2 (\Gamma_n \backslash G)}$ has a spectral gap. Then  
any weak* limit of a subsequence of $(\mu_{\Gamma_n \backslash X, \phi_j^{(n)}})$ is a probability measure on $\mathcal{E}_{s_0}$ and the only possible aperiodic limit is $\mu_{{\rm Gauss}, s_0}$.
\end{conjecture}

\subsection{Probability measures on $\mathcal{D}' (B)$} To conclude this section we raise some general problem related to Conjecture \ref{C1}.

Let $\mu$ be a probability measure on $\mathcal{D}' (B)$ and suppose that it is invariant under the $G$-action given by some representation $\pi_s$. Then $(\mathcal{D}' (B) , \mu)$ is a probability space endowed with a (non free) action of $G$ and the push-forward of $\mu$ under the stabilizer map defines an IRS of $G$. We shall say that the IRS is {\it induced} from $\mu$. 

As an application of the Nevo-Stuck-Zimmer Theorem \cite{StuckZimmer,Nevo} ergodic IRSs in higher rank simple Lie groups are classified, see \cite[Theorem 1.14]{7samurai}. A very natural similar question would be to classify all probability measures on $\mathcal{D}' (B)$ that are invariant and ergodic under the $G$-action given by some representation $\pi_s$. Let us more modestly first describe some families of examples of such measures. 

Let $s \in \Lambda^+$. The standard Gaussian probability measure on $\mathcal{D}' (B)$ is invariant and ergodic under the $G$-action given by the representation $\pi_s$. It moreover follows from \eqref{e:whitecov} applied to $f=g=1$ that if satisfies the following normalisation:
\begin{equation} \label{normalisation}
\int_{\mathcal{D} ' (B)} |T (1)|^2 d\mu_{\rm Gauss} (T) =1.
\end{equation}
There are many other such measures. 

Indeed: let $\Gamma$ be a uniform, torsion free, discrete subgroup of $G$ and let $\phi$ be a $\Gamma$-invariant function in $\mathcal{E}_s$. We have associated to $\phi$ a probability measure $\mu_\phi$ on 
$\mathcal{E}_s$. From this one can get a probability measure on $\mathcal{D}' (B)$. In loose terms we push forward the measure using the inverse of $p_s$. To give a formal construction, one can proceed as follows.

It corresponds to $\phi$ an embedding of $\pi_s$ as a direct summand of $L^2 (\Gamma \backslash G)$. Denote by $\mathcal{H}$ the space of $\pi_s$. Since $L^2 (\Gamma \backslash G)$ is self-dual, the dual representation $(\pi_s ' , \mathcal{H} ')$ also occurs as a direct summand in 
$L^2 (\Gamma \backslash G)$. The inclusion
$$i : \mathcal{H} ' \hookrightarrow L^2 (\Gamma \backslash G)$$
maps the subspace of smooth vectors $\mathcal{H}' {}^{\infty} \subset \mathcal{H} '$ to smooth functions, which can be evaluated at the identity. The evaluation map 
$$t : v \mapsto i(v) (e) \quad (v \in  \mathcal{H}' {}^{\infty})$$
is continuous with respect to the topology of $\mathcal{H}' {}^{\infty}$, and thus defines a distribution vector for $\pi_s$ --- a $\Gamma$-invariant distribution vector in $\mathcal{H}^{-\infty} = \mathcal{D}' (B)$ since $i(v')$ is a $\Gamma$-invariant function. The orbit of this vector under the $\pi_s (G)$-action therefore yields a continuous map 
$$G/\Gamma \to \mathcal{D} ' (B).$$
We denote by $\overline{\mu}_{\Gamma , \phi}$ the push-forward of the normalized Haar measure on $G / \Gamma$ by this map. It defines a probability measure on $\mathcal{D}' (B)$ that is invariant and ergodic under the $G$-action given by the representation $\pi_s$. If moreover $\phi$ is of \emph{normalized $L^2$-norm $1$}, i.e. satisfies 
$$\frac{1}{\mathrm{vol} (\Gamma \backslash G)} \int_{\Gamma \backslash G} |\phi|^2 = 1$$
then $\overline{\mu}_{\Gamma , \phi}$ satisfies \eqref{normalisation} or equivalently:
$$\int_{\mathcal{D} ' (B)} |T (1)|^2 d\overline{\mu}_{\Gamma , \phi} (T) = \int_{\mathcal{E}_s} |\psi (e) |^2 d\mu_{\phi} (\psi ) = \frac{1}{\mathrm{vol} (\Gamma \backslash G)} \int_{G/\Gamma} |\phi (g^{-1})|^2 dg =1.$$

The IRS induced by $\widehat{\mu}_{\Gamma , \phi}$ is supported on the conjugacy class of a lattice. In rank $1$ where one can construct many interesting IRSs, e.g. associated to normal subgroups of a lattice (see \cite{IMRN} for more examples), one can similarly construct measures on $\mathcal{D}' (B)$ from eigenwaves on the corresponding unimodular random $X$-manifolds. 

 A general interesting problem would be to determine the possible weak* accumulation points of families of such measures when the corresponding IRSs BS-converge toward the trivial one. Measures $\mu$ whose induced IRS are trivial indeed correspond to aperiodic limit measures in Conjecture \ref{C1}. 

Any such accumulation point is a measure on $\mathcal{D}' (B)$ that is both invariant and ergodic under the $\pi_s (G)$-action. It is therefore natural to ask for a classification of mean zero, normalized probability measures on $\mathcal{D}' (B)$ that are both invariant and ergodic under the $\pi_s (G)$-action and whose induced IRSs are trivial.

\section{The Quantum Ergodicity Theorem}\label{s:qes1}

In this section, we assume $G$ is of rank $1$ and we take a deterministic point of view to address the question of two-point correlations of eigenfunctions. 
If $\phi_{s_n}$ is an eigenfunction of $L^2$-norm $1$, then we would like to show that in the Benjamini-Schramm limit, and when $s_n \to s$, the correlation function $\phi_{s_n}(x)\phi_{s_n}(y)$ is proportional to $\varphi_s(d(x,y))$, where $\varphi_s$ is the spherical function and by abuse of notation we write $\varphi_s(d(x,y)) = \varphi_s(a_r)$ for $r = d(x,y)$, where $a_r = \exp(r H)$ and $H \in \ag \simeq \R$ is of norm $1$. In other words we want to show that the two-point correlation function of the eigenfunctions converges to the two-point correlation function (or the covariance kernel) of the standard Gaussian wave associated with $s$ (see Definition \ref{d:gauss}). We are able to prove a weak form of this via a quantum ergodicity theorem.

Let $\Gamma_n $ be a uniformly discrete sequence of lattices in $G$ that BS-converges toward the trivial group, and let $A^{(n)} : G \times G \to \R$ be a sequence of kernels satisfying
\begin{equation}\label{e:ker1} \forall \gamma \in \Gamma_n \quad A^{(n)}(\gamma x, \gamma y) = A^{(n)}(x,y)\end{equation}
and
\begin{equation}\label{e:ker2} \forall k_1,k_1 \in K \quad A^{(n)}(xk_1, yk_2) = A^{(n)}(x,y)\end{equation}
for any $x,y \in G$. 
We assume moreover that there exists $M >0$ such that
\begin{equation}\label{e:ker3}
A^{(n)}(x,y) = 0 \quad \text{when} \quad d(x,y) > M.
\end{equation}
This defines an operator $\mathbf{A}^{(n)}$ on $\Gamma_n \backslash X = \Gamma_n \backslash G / K$, by the formula
$$ \mathbf{A}^{(n)} f (x) = \int_G A^{(n)}(x,y) f(y) \, dy,$$
valid for any $f \in C (\Gamma_n \backslash G / K)$.

Let $\{ \phi_j^{(n)} \}$ be an orthogonal basis of Laplacian eigenfunctions on $\Gamma_n \backslash X$. 
We shall see each $\phi_j^{(n)}$ as a function of norm $1$ in $L^2 (\Gamma_n \backslash G)$, with respect to the scalar product 
$$\langle f_1 , f_2 \rangle_{L^2 (\Gamma_n \backslash G)} = \frac{1}{\mathrm{vol} (\Gamma_n \backslash G)} \int_{\Gamma_n \backslash G} f_1 (g) f_2 (g) dg,$$ 
and denote by $is_j^{(n)}$ the parameter (in $i\mathfrak{a}^*$) of the representation it generates in the (quasi-)regular representation $\rho_{\Gamma_n \backslash G}$ in $L^2 (\Gamma_n \backslash G)$. This means that the eigenvalue $\lambda_j^{(n)}$ of the eigenfunction $\phi_j^{(n)}$ is given by
$$ \lambda_j^{(n)} = \rho^2 + (s_j^{(n)})^2,$$
(see Section \ref{s:notation}). In particular if $G = \SL(2,\R)$, then $\rho^2 = \frac14$.
We have the following \emph{quantum ergodicity} theorem.

\begin{theorem}\label{T2}
Assume that the representation of $G$ in $\oplus_n L_0^2 (\Gamma_n \backslash G)$ has spectral gap. Let $M$ be a positive real number and let  $(A^{(n)})_{n \in \mathbb{N}}$ be a uniformly bounded sequence of kernels on $G\times G$ satisfying \eqref{e:ker1}, \eqref{e:ker2} and \eqref{e:ker3}. Fix $is_0 \in \Lambda^+$. There exists a sequence $(\delta_n )$ of positive real numbers converging to $0$ such that letting  
$$N(\delta_n , \Gamma_n ) = \# \{ j \; : \; s_j^{(n)} \in [s_0 -\delta_n , s_0 + \delta_n] \} .$$
and $I_n= I_n(s_0) :=  [s_0 -\delta_n , s_0 + \delta_n] $, we have:
\begin{equation}\label{e:QEvariance}
\frac{1}{N(\delta_n , \Gamma_n )} \sum_{s_j^{(n)} \in I_n } \left| \langle  \phi_j^{(n)} , \mathbf{A}^{(n)} \phi_j^{(n)} \rangle_{L^2 (\Gamma_n \backslash G )} - \langle \mathbf{A}^{(n)} \rangle_{s_0}   \right|^2 \to 0
\end{equation}
as $n$ tends to infinity, where 
$$  \langle \mathbf{A}^{(n)} \rangle_{s} = \frac1{\vol(\Gamma_n\backslash G)} \int_{\Gamma_n\backslash G} \int_G A^{(n)}(x,y) \varphi_s(x^{-1}y) \, dxdy,$$
is the average of the kernel $A^{(n)}$ against the spherical function of spectral parameter $s$.
\end{theorem} 

\begin{unremark}
In the theorem the sequence $(\delta_n)$ depends on $(A^{(n)})_{n \in \mathbb{N}}$. However the proof will show (and even rely on the fact)  that if $\delta$ is any positive real number such that the interval $I_\delta =[s_0 - \delta , s_0 + \delta]$ is contained in the tempered spectrum, i.e. $i I_\delta \subset\Lambda^+$, then 
\begin{equation} \label{E:not-shrinking}
\frac{1}{N(\delta , \Gamma_n )} \sum_{s_j^{(n)} \in I_\delta } \left| \langle  \phi_j^{(n)} , \mathbf{A}^{(n)} \phi_j^{(n)} \rangle_{L^2 (\Gamma_n \backslash G )} - \langle \mathbf{A}^{(n)} \rangle_{s_j^{(n)}}   \right|^2 \to 0
\end{equation}
as $n$ tends to infinity; see the proof of Lemma \ref{t:shrinking window}.

In the proof of Theorem \ref{T2} we disintegrate the kernel $A^{(n)}$ into components $A_r^{(n)} \in \mathcal{D}'(G\times G)$ supported on $(x,y) \in G \times G$ such that $d(x,y) = r$, and prove \eqref{e:QEvariance} for each component. For $r=0$, the operator $\mathbf{A}_r^{(n)}$ is simply the multiplication by a function $a^{(n)} : G\to \R$ and this gives
\begin{equation*}
\frac{1}{N(\delta_n , \Gamma_n )} \sum_{s_j^{(n)} \in I_n } \left| \langle  \phi_j^{(n)} , a^{(n)} \phi_j^{(n)} \rangle_{L^2 (\Gamma_n \backslash G )} - \frac1{\vol(\Gamma_n\backslash G)} \int_{\Gamma_n\backslash G} a^{(n)}(g) \, dg   \right|^2 \to 0,
\end{equation*}
when $n\to +\infty$. Taking a fixed positive $\delta$ as in \eqref{E:not-shrinking} we get 
\begin{equation*}
\frac{1}{N(\delta , \Gamma_n )} \sum_{s_j^{(n)} \in I_\delta } \left| \langle  \phi_j^{(n)} , a^{(n)} \phi_j^{(n)} \rangle_{L^2 (\Gamma_n \backslash G )} - \frac1{\vol(\Gamma_n\backslash G)} \int_{\Gamma_n\backslash G} a^{(n)}(g) \, dg   \right|^2 \to 0.
\end{equation*}
In particular on recovers the Quantum Ergodicity theorem for hyperbolic surfaces of \cite{LMS} where only multiplication by functions are considered.
\end{unremark}

\section{A few simplifications}\label{s:qes2}
To prove Theorem \ref{T2} we first proceed by a series of reductions at a fixed level. We will therefore drop the index $n$ in what follows, adding it back only when needed. 
Note that
\begin{multline}\label{e:QEvariance2}
\sum_{s_j \in I } \left| \langle  \phi_j , \mathbf{A} \phi_j \rangle_{L^2 (\Gamma \backslash G )} - \langle \mathbf{A} \rangle_{s_0}   \right|^2 \\
 \lesssim  \sum_{s_j \in I } \left| \langle  \phi_j, \mathbf{A} \phi_j \rangle_{L^2 (\Gamma \backslash G )} - \langle \mathbf{A} \rangle_{s_j}   \right|^2 +  \sum_{s_j \in I } \left|  \langle \mathbf{A} \rangle_{s_j} -  \langle \mathbf{A} \rangle_{s_0}   \right|^2.
\end{multline}
We will bound the two sums on the right-hand side separately. The first one is the main part of the proof, and we will need ergodic theory to estimate it. The second involves only a spectral density estimate. Let us first introduce a few useful objects.

\subsection{Disintegration and radial averages}\label{s:disintegration}
We define the operator $\mathbf{A}_r$ acting on functions $f \in C(G/K)$ by
$$ \mathbf{A}_r f(x) = \int_K A(x, x k a_r ) f(x k a_r) dk. $$
It can be seen as a radial disintegration of $\mathbf{A}$ such that we have
\begin{equation}\label{e:disintegration}
 \mathbf{A} = \int_0^M \mathbf{A}_r \, \sinh(\rho r)dr,
 \end{equation}
where $\rho$ is defined in Section \ref{s:notation}, and here $\rho \in \R_+$ because we are in rank $1$. Because of \eqref{e:ker3}, the map $r \mapsto [A]_r$ is compactly supported in $[0,M]$.
We also define an average of the kernel $A$ over geodesic segments of length $r$
\begin{equation}\label{e:average simple}
[ A ]_r =  \frac1{\vol(\Gamma\backslash G)} \int_{\Gamma\backslash G} A(x, x a_r) \, dx,
\end{equation}
and we denote by $[ \mathbf{A} ]$ the convolution operator with radial kernel $r \mapsto [ A ]_r $, i.e.
$$ [ \mathbf{A} ] f(x) = \int_{G/K} [ A ]_{d(x,y)} f(y) \, dy.$$

As $G$ is of rank $1$, via the decomposition $g = k_1 \exp_G(H) k_2$ with $H \in \ag \simeq \R$ and $k_1,k_2\in K$ we can see a radial kernel such as $[A]_r$ as a function on $K \backslash G / K$. Recall that if $F \in C_c^\infty (K \backslash G / K)$ then, for every $is \in i\mathfrak{a}^*$, the spherical function $\varphi_s$ is an eigenfunction of the convolution operator associated with $F$. The {\it spherical transform} $\widehat{k} (s )$ is defined as the corresponding eigenvalue, i.e. 
\begin{equation}\label{e:sphertrans}
F * \varphi_s = \widehat{F} (s ) \varphi_s .
\end{equation}
More generally any eigenfunction of eigenvalue $\lambda$ is an eigenfunction of the convolution by $F$, with eigenvalue $\hat F (s)$ where we recall the parametrization $\lambda = \rho^2 + s^2$. The spherical transform extends to an $L^2$-isometry between $L^2 (K \backslash G / K)$ and $L^2 (\mathfrak{a}^* , \mu_{\rm Planch})$ for an appropriate (and explicit, see \cite{HC}) $W$-invariant measure $\mu_{\rm Planch}$ on $\mathfrak{a}^*$, the {\it Plancherel measure}, which is absolutely continuous with respect to the Haar measure.

Let us now record a few useful properties.
\begin{lem}\label{l:average}
We have $$[[\mathbf{A}]] = [\mathbf{A}]$$ and if $[A]$ is the kernel of $[\mathbf{A}]$, that is $[A](x,y) = [A]_{d(x,y)}$, then
$$ [[A]]_r = [A]_r.$$
Moreover the two quantities $\langle \mathbf{A} \rangle_{s_j}$ and $[\mathbf{A}]$ are related by the formula
$$\langle \mathbf{A} \rangle_{s_j} =  \langle \phi_j, [\mathbf{A}] \, \phi_j \rangle,$$
for any $L^2$-normalized eigenfunction $\phi_j$ of eigenvalue $\lambda_j = \rho^2 + s_j^2$.
\end{lem}
\begin{proof}
The first two properties can be checked easily. Let us prove that
$$\langle \mathbf{A} \rangle_{s_j} =  \langle \phi_j, [\mathbf{A}] \, \phi_j \rangle.$$
We know that $\phi_j$ is an eigenfunction of $[\mathbf{A}]$ with eigenvalue given by the spherical transform of the radial kernel $r \mapsto [A]_r$ evaluated at $s_j$. By definition of the spherical transform, this eigenvalue is equal to 
$$ \int_0^M [A]_r \, \varphi_{s_j}(r) \, \sinh(\rho r) dr. $$
As $\phi_j$ is $L^2$-normalized we have
\begin{align*}
 \langle \phi_j, [\mathbf{A}] \, \phi_j \rangle = \int_0^M [A]_r \, \varphi_{s_j}(r) \, \sinh(\rho r) dr.
\end{align*}
We now use the expression of $[A]_r$ and the $K$-invariance of $A$
\begin{align*}
 \langle \phi_j, [\mathbf{A}] \, \phi_j \rangle &=\int_0^M \frac1{\vol(\Gamma\backslash G)} \int_{\Gamma\backslash G} A(x, x a_r) \, dx \, \varphi_{s_j}(r) \, \sinh(\rho r) dr\\
  &=\int_0^M \int_K \int_K \frac1{\vol(\Gamma\backslash G)} \int_{\Gamma\backslash G} A(x k_1^{-1}, x a_r k_2) \, dx \, \varphi_{s_j}(r) \, \sinh(\rho r) dk_1 dk_2 dr\\
   &=\int_0^M \int_K \int_K \frac1{\vol(\Gamma\backslash G)} \int_{\Gamma\backslash G} A(x, x k_1 a_r k_2) \, dx \, \varphi_{s_j}(r) \, \sinh(\rho r) dk_1 dk_2 dr\\
  &=\int_G \frac1{\vol(\Gamma\backslash G)} \int_{\Gamma\backslash G} A(x,xy) \, dx \, \varphi_{s_j}(y) \, dy
\end{align*}
By a last change of variable $y \to xy$ and rearrangement of the integrals we obtain
$$ \langle \phi_j, [\mathbf{A}] \, \phi_j \rangle =  \frac1{\vol(\Gamma\backslash G)} \int_{\Gamma\backslash G} \int_G A(x,y)  \varphi_{s_j}(x^{-1}y) \, dx dy $$
as required.
\end{proof}

\subsection{Two simplifications}
What we just introduced allows us to make some simplifications. According to \eqref{e:QEvariance2}, we need to bound
$$\sum_{s_j \in I } \left| \langle  \phi_j , \mathbf{A} \phi_j \rangle_{L^2 (\Gamma \backslash G )} - \langle \mathbf{A} \rangle_{s_j}   \right|^2.$$
We will instead assume that $[\mathbf{A}] = 0$ and estimate
$$ \sum_{j: s_j \in I_n} \left| \langle \phi_j, \mathbf{A} \phi_j \rangle\right|^2.$$
Indeed we can then apply the estimate to
$$ \mathbf{B} = \mathbf{A} - [\mathbf{A}]$$
as by Lemma \ref{l:average} we have  $[\mathbf{B}] = [\mathbf{A}] - [[\mathbf{A}]] = 0$, and
$$ \sum_{j : s_j \in I_n} \left| \langle \phi_j , (\mathbf{A} - [\mathbf{A}]) \phi_j \rangle\right|^2 =  \sum_{j : s_j \in I_n} \left| \langle \phi_j , \mathbf{A} \phi_j \rangle - \langle \mathbf{A} \rangle_{s_j} \right|^2. $$
 By uniqueness of the kernel, our assumption $[\mathbf{A}] = 0$ means that
 \begin{equation}\label{e:zero average}
 \forall r\geq 0 \quad [A]_r = 0.
 \end{equation}

The second simplification is to use the disintegration \eqref{e:disintegration} in order to write
\begin{align*}
 \sum_{j: s_j \in I_n} \left| \langle \phi_j, \mathbf{A} \phi_j \rangle \right|^2
 & =  \sum_{j: s_j \in I_n} \left| \langle \phi_j , \left( \int_0^M \mathbf{A}_r \, \sinh(\rho r) dr \right) \phi_i \rangle \right|^2\\
 & =  \sum_{j : s_j \in I_n} \left| \int_0^M \langle \phi_j ,  \mathbf{A}_r \phi_j \rangle \sinh(\rho r) dr\right|^2\\
 &\leq M \sinh^2( \rho M) \int_0^M \left(\sum_{j : s_j \in I_n} \left| \langle \phi_j , \mathbf{A}_r \phi_j \rangle \right|^2\right) \, dr,
\end{align*}
where the last line is obtained by the Cauchy-Schwarz inequality.  This tells us that it is sufficient to estimate the term between brackets in the last line, as $M$ is fixed.

In conclusion, we have reduced the bound of the first term on the left-hand side of \eqref{e:QEvariance2} to bounding
$$\sum_{j : s_j \in I_n} \left| \langle \phi_j , \mathbf{A}_r \phi_j \rangle \right|^2,$$
assuming that $[A]_r = 0$. 

\subsection{Spectral averages}
Before estimating this term, let us look at the second term in \eqref{e:QEvariance2}, namely
$$ \sum_{s_i \in I } \left|  \langle \mathbf{A} \rangle_{s_j} -  \langle \mathbf{A} \rangle_{s_0}   \right|^2.$$
We denote by $F : K\backslash G / K \to \R$ the function corresponding to the radial kernel $r \mapsto [A]_r$. 
We have $$ \widehat F(s) =  \langle \mathbf{A} \rangle_{s}, $$ where $\widehat F$ is the spherical transform of $F$ (see Section \ref{s:disintegration}, in particular Lemma \ref{l:average}).
Reintroducing the index $n$ what we want to show is therefore that
$$
\frac{1}{N(\delta_n , \Gamma_n )} \sum_{j \; : \; s_j^{(n)} \in [s_0 -\delta_n , s_0 + \delta_n] } \left| \widehat{F}_n (s_j^{(n)}) -  \widehat{F}_n(s_0 ) \right|^2 \to 0.
$$
This is the content of the following lemma.

\begin{lem}\label{t:shrinking window}
Fix $is_0 \in \Lambda^+$. There exists a sequence $(\delta_n )$ of positive real numbers (depending on the sequence of test kernels $\mathbf{A}^{(n)}$) that converges to $0$ and satisfies the following  properties.
\begin{enumerate}
\item We have:
\begin{equation} \label{e:a=1}
\frac{1}{N(\delta_n , \Gamma_n )} \sum_{j \; : \; s_j^{(n)} \in [s_0 -\delta_n , s_0 + \delta_n] } \left| \widehat{F}_n (s_j^{(n)}) -  \widehat{F}_n (s_0 ) \right|^2 \to 0
\end{equation}
as $n$ tends to infinity. 
\item For all $n \geq 1$, we have $\delta_n \geq r_n^{-1}$. 
\end{enumerate}
\end{lem}
\begin{proof} By hypothesis $s_0$ correspond to an eigenvalue in the interior of the $L^2$-spectrum of $X$. Let $\chi$ be a smooth cutoff function supported in $[-1 , 1]$ and taking the constant value $1$ on $[-1/2 , 1/2]$. Identifying $\mathfrak{a}^*$ with $\R$ (recall that we suppose that $G$ is of rank 1), we shall see $\chi$ as a function on $\mathfrak{a}^*$. 

Now let $\varepsilon$ be a positive real number such that $[s_0 - 2\varepsilon , s_0 + 2\varepsilon]$ is contained in the $L^2$-spectrum of $X$. We define $\widehat \chi_{\varepsilon} (s ) = \chi \left(\frac{1}{2\varepsilon} (s -s_0) \right)$; it is a function equal to $1$ in the spectral interval $[s_0 -\varepsilon  , s_0 + \varepsilon]$ and compactly supported in $[s_0 -2\varepsilon , s_0 + 2\varepsilon]$.  Being compactly supported, $\widehat \chi_{\varepsilon}$ is the spherical transform of a rapidly decaying function $\chi_{\varepsilon}$ on $G$. 

Up to replacing each $F_n$ with $F_n- \frac{\widehat{F}_n (s_0 )}{\widehat{\chi}_{\varepsilon} (s_0 )} \chi_{\varepsilon}$, we may assume that $\widehat{F}_n (s_0 ) = 0$. The radial functions $F_n$ are uniformly bounded and supported in the ball of radius $M$, their spherical transforms $\widehat{F}_n$ therefore belong to a compact subspace $C$ of test functions on $\mathfrak{a}^*$.

Now \cite[Theorem 1.2]{7samurai} implies that the spectral measure of $\Gamma_n \backslash X$ weakly converges toward the spectral measure of $X$. For all positive $\delta$ such that 
$[s_0 - \delta , s_0 + \delta]$ is contained in the tempered spectrum, and for any test function $\widehat{F}$ on $\mathfrak{a}^*$, we have
$$\frac{1}{\mathrm{vol} (\Gamma_n \backslash G)}  \sum_{j \; : \; s_j^{(n)} \in [s_0 -\delta , s_0 + \delta] } \left| \widehat{F} (s_j^{(n)})  \right|^2 \to  \int_{[s_0 - \delta , s_0 + \delta]} \left| \widehat{F} (s) \right|^2 p(s) ds,$$
as $n$ tends to infinity, where the limit is uniform for $\widehat{F} \in C$. In particular we have
$$\frac{1}{\mathrm{vol} (\Gamma_n \backslash G)}  \sum_{j \; : \; s_j^{(n)} \in [s_0 -\delta , s_0 + \delta] } \left| \widehat{F}_n (s_j^{(n)})  \right|^2 \to  \int_{[s_0 - \delta , s_0 + \delta]} \left| \widehat{F}_n (s) \right|^2 p(s) ds,$$
as $n$ tends to infinity.
Since on the other hand 
$$\frac{N(\delta,\Gamma_n)}{\mathrm{vol} (\Gamma_n \backslash G)} \to  \int_{[s_0 - \delta , s_0 + \delta]} p(s) ds$$
as $n$ tends to infinity and the Plancherel density $p$ is not zero at $s_0$, we conclude that 
$$\frac{1}{N(\delta,\Gamma_n)}  \sum_{j \; : \; s_j^{(n)} \in [s_0 -\delta , s_0 + \delta] } \left| \widehat{F}_n (s_j^{(n)})  \right|^2 \to \frac{\int_{[s_0 - \delta , s_0 + \delta]} \left| \widehat{F}_n (s) \right|^2 p(s) ds}{ \int_{[s_0 - \delta , s_0 + \delta]} p(s) ds},$$
as $n$ tends to infinity. The sequence $(\delta_n)$ where $\delta_n$ is the infimum over all positive $\delta$ such that  
$$\frac{1}{N(\delta,\Gamma_n)}  \sum_{j \; : \; s_j^{(n)} \in [s_0 -\delta , s_0 + \delta] } \left| \widehat{F}_n (s_j^{(n)})  \right|^2 \leq 2 \frac{\int_{[s_0 - \delta , s_0 + \delta]} \left| \widehat{F}_n (s) \right|^2 p(s) ds}{ \int_{[s_0 - \delta , s_0 + \delta]} p(s) ds}$$ 
and 
$$\delta_n \geq r_n^{-1},$$
finally satisfies the desired properties.  
\end{proof}

\section{Main estimate}\label{s:qes4}
Recall that we need to bound
\begin{equation}\label{e:tobound}
\sum_{j : s_j \in I_n} \left| \langle \phi_j , \mathbf{A}_r \phi_j \rangle \right|^2,
\end{equation}
assuming that $[A]_r = 0$. 

Following the proof of \cite{LMS}, an essential ingredient is a wave propagation operator that we generalize here.
Let $k_t = \textbf{1}_{B_t} / \sqrt{\vol (B_t )}$, where 
$$B_t = \{ k a_r k' \; : \; k, k' \in K, r \leq t  \},$$
and $\textbf{1}_{B_t}$ is the characteristic function of the set $B_t$.
The convolution operator $\rho_{\Gamma \backslash G} (k_t)$ can be roughly seen as a wave propagator at time $t$.

We want to replace \eqref{e:tobound} with 
$$\sum_{j : s_j \in I_n} \left| \langle \phi_j  ,\frac1T \int_0^T \waveker \mathbf{A}_r \waveker \,dt \, \phi_j \rangle \right|^2$$
in order to take advantage of the ergodic properties arising from our spectral gap assumption. For this we first need to look at the action of $\waveker$ on eigenfunctions.

\subsection{Spectral side}We know that if $\psi$ is an eigenfunction associated with the eigenvalue $\lambda = \rho^2 + s^2$,
$$\rho_{\Gamma \backslash G} (k_t) \psi = h_t(s) \psi,$$
where $h_t(s) = \widehat k_t(s)$ is the spherical transform of $k_t$ (see \eqref{e:sphertrans}).
Since $\varphi_s (e) = 1$, it follows from the definition of the spherical transform that we have:
$$h_t (s) = \frac{1}{\sqrt{\vol (B_t )}} \int_{B_t} \varphi_{s} (g) \, dg.$$
The following lemma is classical, see e.g. \cite[Eq. (4.6.14)]{GangV}.
\begin{lem}
Let $H$ in $\mathfrak{a}$ be so that $\rho (H)>0$. Set $a_t = \exp (tH)$. There exists some positive real number $\delta$ such that for every positive real number $\varepsilon$, there exists a constant $C_\varepsilon$ (that depends only on $\varepsilon$ and $H$) such that for all real $s$ with $|s | \geq \varepsilon >0$ we have
$$\| e^{\rho (t H)} \varphi_{s} (a_t ) - B (is) (e^{-is \rho (tH )} + e^{is \rho (tH) }) \| \leq C_\varepsilon e^{-\delta t}.$$
\end{lem}
Here $B$ is some explicit analytic function, products of Gamma functions. Taking $H$ to be of norm $1$ we get:
\begin{equation*}
\begin{split} 
\int_{B_t} \varphi_{s} (g) dg & = \int_0^t \sinh (2 \rho u) \varphi_{s} (a_u ) du \\
& = 2B(is) \int_\varepsilon^t \cos (s \rho u ) e^{\rho u } du + O (e^{(\rho - \delta ) t})
\end{split}
\end{equation*}
as $t$ tends to infinity. It follows that for $t$ large 
$$\frac{1}{\sqrt{\vol (B_t )}} \int_{B_t} \varphi_{is} (g) dg$$
is close to 
$$\mathcal I_{t,s} = \frac{2B(is)}{e^{\rho t}} \int_\varepsilon^t \cos (s \rho u ) e^{\rho u } du.$$
By double integration by parts we find that
$$\mathcal I_{t,s} = \frac{2B(is)}{(1+s^2)\rho} \left( \cos(s\rho t) + s \sin(s\rho t)\right) + O(e^{-\rho t}).$$ 
and adding small intervals around $t_k = \frac{2\pi k}{s\rho}$ on which $\mathcal I_{t,s}$ is uniformly bounded from below (as is detailed in \cite[Section 8]{LMS}), we get that there exists a constant $C_s$ such that for any $T$ large enough
$$\frac{1}{T} \int_0^T |h_t(s) |^2 dt \geq C_s.$$
This can be done uniformly for $s \in I$.

As
$$  \langle  \phi_j , \mathbf{A}_r \, \phi_j \rangle  = \frac1{\frac1T\int_0^T |h_t(s)|^2 dt} \langle  \phi_j ,  \frac1T \int_0^T \rho_{\Gamma \backslash G} (k_t ) \,\mathbf{A}_r \, \rho_{\Gamma \backslash G} (k_t) \, dt  \, \phi_j \rangle,$$
we can now write
\begin{align*}
\sum_{j : s_j \in I} \left| \langle \phi_j, \mathbf{A}_r \phi_j \rangle \right|^2 
&= O_{I} \left( \sum_{j : s_j \in I_n} \left| \langle \phi_j, \frac1T \int_0^T \waveker \mathbf{A}_r \waveker \,dt \, \phi_j \rangle \right|^2 \right) \\
&= O_{I} \left(  \left\|  \frac1T \int_0^T \waveker \mathbf{A}_r \waveker \,dt \right\|_\text{HS}^2 \right),
\end{align*}
with an implied constant depending only on the interval $I$.
Estimating the Hilbert-Schmidt norm of this time average constitutes what we call the Geometric side of the proof.
\subsection{Geometric side}
We first show that the kernel whose Hilbert-Schmidt norm we want to compute can be expressed as the convolution of a function $b_r$, where we define
$$ b_r(g) = A(g,ga_r).$$
This is contained in the following lemma.
\begin{lem}
The kernel of the operator 
$$\frac{1}{T} \int_0^T \rho_{\Gamma \backslash G} (k_t ) \mathbf{A}_r \rho_{\Gamma \backslash G} (k_t ) dt $$
acting on $L^2 (\Gamma \backslash G / K)$ is 
$$\sum_{\gamma \in \Gamma} F(g , \gamma h)$$
where $F : G\times G \to \R$ satisfies the invariance properties \eqref{e:ker1} and \eqref{e:ker2}, and there exist a measurable function $\phi_{h,t,r}: G\to \R$ and a constant $m_{h,t,r} > 0$ such that
$$ F(g,h) =  \frac{1}{T} \int_0^T m_{g^{-1}h,t,r}  \rho_{\Gamma\backslash G} (\phi_{g^{-1}h,t,r}) b_r(g) dt, \quad g,h \in G.$$
We have more precisely
$$\phi_{h,t,r} = \frac{1}{\vol (B_t \cap hB_t a_{-r})} \mathbf{1}_{B_t \cap hB_t a_{-r}},$$
and
$$ m_{h,t,r} =  \frac{\vol (B_t \cap hB_t a_{-r})}{\vol (B_t)}$$ 
In particular $F(g,h)=0$ whenever $d(g,h) \geq 2T + r$. 
\end{lem}
\begin{proof}
We obtain by a simple computation and application of Fubini's theorem
\begin{align*}
&\rho_{\Gamma \backslash G} (k_t ) \mathbf{A}_r \rho_{\Gamma \backslash G} (k_t ) f(g) \\
&\quad = \frac1{\vol(B_t)} \int_K \int_{gB_t} \int_{xka_r B_t}  A(x,xka_r) f(h) \, dh \, dx \, dk \\
&\quad = \frac1{\vol(B_t)} \int_{G}  \left(\int_K \int_{gB_t\cap h B_t a_{-r} k^{-1}} A(x,xka_r) \, dx \, dk \right)\, f(h) \, dh
\end{align*}
So we have
$$F (g,h) = \frac{1}{T} \int_0^T \frac{1}{\vol (B_t )} \int_K \int_{gB_t \cap h B_t a_{-r} k^{-1}} A (x,xka_r) \, dx dk dt.$$
Doing a change of variable $x \mapsto xk$ we obtain
$$F (g,h) = \frac{1}{T} \int_0^T \frac{1}{\vol (B_t )} \int_K \int_{gB_tk \cap h B_t a_{-r}} A (xk^{-1},xa_r) \, dx dk dt.$$
We then use that $gB_t k = g B_t$ by definition of $B_t$ and that $A$ is right $K$-invariant to get
$$F (g,h) = \frac{1}{T} \int_0^T \frac{1}{\vol (B_t )} \int_{gB_t \cap h B_t a_{-r}} A (x,xa_r) \, dx dt.$$
By the change of variable $x \mapsto g^{-1}x$ we then have
\begin{align*}
F (g,h) &= \frac{1}{T} \int_0^T \frac{1}{\vol (B_t )} \int_{B_t \cap g^{-1} h B_t a_{-r}} A (gx,gxa_r) \, dx dt\\
&= \frac{1}{T} \int_0^T m_{g^{-1}h,t,r}  \rho_{\Gamma\backslash G} (\phi_{g^{-1}h,t,r}) b_r(g) dt.
\end{align*}

Now to see that $F(g,h)=0$ when $d(g,h) \geq 2T + r$ we note that 
$B_t \cap g^{-1}hB_t a_{-r}$
is the intersection of a ball of radius $t$ centered at $e$ and of a ball of radius $t+r$ centered at $g^{-1}h$. The intersection is empty when $g^{-1}h$ is at distance greater than $2t+r$ from $e$, or in other words if $d(g,h) \geq 2t + r$. By integrating over $t\in[0,T]$ we obtain that this intersection is empty whenever $d(g,h) \geq 2T + r$, in which case $F(g,h) = 0$.
\end{proof}

We therefore have:
\begin{equation*}
\begin{split}
\left\| \frac{1}{T} \int_0^T \rho_{\Gamma \backslash G} (k_t )\, \mathbf{A}_r \,\rho_{\Gamma \backslash G} (k_t ) \, dt \right\|_{\rm HS}^2 & = \int_{\Gamma \backslash G} \int_{\Gamma \backslash G} \left| \sum_{\gamma \in \Gamma} F(g , \gamma h) \right|^2 dg dh \\
& \leq \int_{\Gamma \backslash G} \int_G |F(g,h)|^2 dg dh \\
&\quad+ \int_{(\Gamma \backslash G)_{<2T+r}} \int_{\Gamma \backslash G} \left| \sum_{\gamma \in \Gamma} F(g , \gamma h) \right|^2 dg dh,
\end{split}
\end{equation*}
where in the second line we split between the points in $\Gamma \backslash G$ whose injectivity radius is $\geq 2T + r$ and those whose injectivity radius is $< 2T + r$. In the first case the sum is reduced to one term because $F(g,h)=0$ when $d(g,h) \geq 2T + r$. We can therefore take the sum out of the absolute value and this yields the integral over $G$.
Since the injectivity radius of $\Gamma \backslash G$ is assumed to be bounded away from $0$ by a uniform constant, the number of terms in the second sum is bounded by a constant times the volume of a ball of radius $2T+r$. It follows from this estimate on the number of terms and the Cauchy-Schwarz inequality that
\begin{align*}
\int_{\Gamma \backslash G} \left| \sum_{\gamma \in \Gamma} F(g , \gamma h)  \right|^2\, dh &\leq \vol (B_G (e, 2T+r)) \int_{\Gamma \backslash G} \sum_{\gamma \in \Gamma} |F(g , \gamma h)|^2 \, dh\\
&\leq  \vol (B_G (e, 2T+r))  \int_{G} |F(g , h)|^2 \, dh\\
&\leq \vol (B_G (e, 2T+r))^2 \|b_r\|_\infty^2.
\end{align*}
 Hence we have
\begin{multline} \label{EBS}
\left\| \frac{1}{T} \int_0^T \rho_{\Gamma \backslash G} (k_t )\, \mathbf{A}_r \,\rho_{\Gamma \backslash G} (k_t ) \, dt \right\|_{\rm HS}^2 
\\ \leq \int_{\Gamma \backslash G} \int_G |F(g,h)|^2 dg dh + O \left( \vol (B_G (e, 2T+r))^2 \vol (\Gamma \backslash G)_{<2T+r} \|b_r\|_\infty^2 \right) .
\end{multline}
Here the implied constants depend on $T$ and $r$ but not on $\Gamma$ (as long as the injectivity radius of $\Gamma \backslash G$ is bounded away from $0$ independently of $\Gamma$).

The change of variables $h \mapsto g^{-1}h$ then leads to:
\begin{equation*}
\begin{split}
\int_{\Gamma \backslash G} \int_G |F(g,h)|^2 dg dh & = \int_{\Gamma \backslash G} \int_{B_{2T+r}} \left|  \frac{1}{T} \int_0^T m_{g^{-1}h,t,r}  \rho_{\Gamma\backslash G} (\phi_{g^{-1}h,t,r}) b_r(g) dt\right|^2 dg dh \\
& = \int_{\Gamma \backslash G} \int_{B_{2T+r}} \left|  \frac{1}{T} \int_0^T m_{h,t,r}  \rho_{\Gamma\backslash G} (\phi_{h,t,r}) b_r(g) dt \right|^2 dg dh.
\end{split}
\end{equation*}
Writing $h=k_1a_uk_2$ the previous expression becomes
$$\int_{\Gamma \backslash G} \int_{K} \int_K \int_{0}^{2T+r} \left|  \frac{1}{T} \int_{((u-r)/2)_+}^T m_{k_1 a_u k_2,t,r}  \rho_{\Gamma\backslash G} (\phi_{k_1 a_u k_2,t,r}) b_r(g) dt \right|^2 dg dk_1 dk_2 \sinh (\rho u) du$$
where $((u-r)/2)_+ = \max (0,(u-r)/2)$.

\medskip

Now a general theorem of Nevo \cite[Theorem 4.1]{GNevo} implies the following lemma that we shall apply to the representation $\rho_{\Gamma \backslash G}$ in $L^2_0 (\Gamma \backslash G)$. This theorem is key to our quantum ergodicity result and on a dynamical level expresses the exponential mixing of the geodesic flow, that we can control thanks to the spectral gap. 

\begin{theorem}\label{l:nevo}
Let $\pi$ be a unitary representation of $G$ with a spectral gap and no invariant vectors. Then, there exists a positive constant $\theta$ (that depends only on the spectral gap) such that 
$$\|\pi (\mathbf{1}_{E}/\vol(E)) \| \leq  \vol (E)^{-\theta},$$
for any measurable subset $E \subset G$.
\end{theorem}
The theorem applied to our setting says that for any $\Gamma$-periodic function $f$ on $G$ such that
$ \int_{\Gamma \backslash G} |f(g)|^2 dg < \infty$, the ``ergodic'' average
$$\rho_{\Gamma \backslash G}(\mathbf{1}_{E}/\vol(E)) f(g) = \frac1{\vol(E)} \int_{gE} f(h) \, dh $$
converges in $L^2$ towards $$ \frac1{\vol(\Gamma \backslash G)} \int_{\Gamma \backslash G} f(g) \, dg,$$
exponentially fast when $\vol(E) \to +\infty$ (i.e. at a rate $\vol(E)^{-\theta}$, where $\theta$ is determined by the spectral gap of the Laplacian).

The sets $B_t$ and $k_1 a_u k_2 B_t a_{-r}$ are balls of radius $t$ and $t+r$ whose center are at distance $u$. Their intersection is bounded by a ball of radius $t-(u-r)_+/2$. We thus have 
$$\vol (B_t \cap k_1 a_u k_2 B_t a_{-r}) \lesssim \vol (B_{t-(u-r)_+/2}) \lesssim e^{2\rho(t-(u-r)_+/2)},$$ 
where $\rho > 0$ depends only on $G$.
In particular
$$ m_{k_1 a_u k_2,t,r} \lesssim \frac{\vol (B_{t-(u-r)_+/2})}{\vol(B_t)} \lesssim e^{-\rho (u-r)_+}.$$
As we assumed in \eqref{e:zero average} that $[A]_r = 0$, we have $b_r \in L^2_0 (\Gamma \backslash G)$, i.e.
$$ \int_{\Gamma\backslash G} b_r(g) \, dg = \int_{\Gamma\backslash G} A(g,ga_r) \,dg = 0,$$
 and by Theorem \ref{l:nevo}
$$\| \rho_{\Gamma\backslash G} (\phi_{k_1 a_u k_2,t,r}) b_r \|_{L^2_0 (\Gamma \backslash G)} \lesssim e^{-2\theta \rho(t- (u-r)_+/2)}   \|b_r \|_{L^2 (\Gamma \backslash G)}. $$
Now applying Minkowski integral inequality to 
$$ \int_{0}^{2T+r} \int_{K} \int_K \int_{\Gamma \backslash G} \frac{1}{T^2} \left| \int_{((u-r)/2)_+}^T m_{k_1 a_u k_2,t,r}  \rho_{\Gamma\backslash G} (\phi_{k_1 a_u k_2,t,r}) b_r(g) dt  \right|^2 dg dk_1 dk_2 \sinh (\rho u) du$$
we obtain
\begin{equation*}
\begin{split}
&\int_{\Gamma \backslash G} \int_G |F(g,h)|^2 dg dh \\
&\quad \leq  \int_{0}^{2T+r} \int_{K} \int_K  \frac{1}{T^2} \left(  \int_{((u-r)/2)_+}^T m_{k_1 a_u k_2,t,r} \| \rho_{\Gamma\backslash G} (\phi_{k_1 a_u k_2,t,r}) b_r \|_{L^2_0 (\Gamma \backslash G)} dt \right)^2  dk_1 dk_2 \sinh (2\rho u) du \\
&\quad \leq  \int_{0}^{2T+r} \int_{K} \int_K  \frac{1}{T^2} \left( \int_{((u-r)/2)_+}^T e^{-\rho (u-r)_+} e^{-2\theta \rho(t- (u-r)_+/2)}   \|b_r \|_{L^2 (\Gamma \backslash G)} dt \right)^2  dk_1 dk_2 \sinh (2\rho u) du \\
&\quad \leq  \frac{ \|b_r \|^2_{L^2 (\Gamma \backslash G)}}{T^2}  \int_{0}^{2T+r} e^{-2\rho(1-\theta) (u-r)_+} \left( \int_{((u-r)/2)_+}^T  e^{-2\theta \rho t}   dt \right)^2  \sinh (2\rho u) du\\
&\quad \leq  \frac{ \|b_r \|^2_{L^2 (\Gamma \backslash G)}}{T^2\theta^2}  \int_{0}^{2T+r} e^{-2\rho(1-\theta) (u-r)_+}  e^{-2\theta \rho (u-r)_+}  \sinh (2\rho u) du\\
&\quad \lesssim e^{2\rho r} \frac{\|b_r \|_{L^2 (\Gamma \backslash G)}^2}{T\theta^2} .
\end{split}
\end{equation*}
Inserting the last inequality into \eqref{EBS} we conclude that
$$\left\| \frac{1}{T} \int_0^T \rho_{\Gamma \backslash G} (k_t )\, \mathbf{A}_r \,\rho_{\Gamma \backslash G} (k_t ) \, dt \right\|_{\rm HS}^2 
\lesssim_r \frac{\|b_r \|_{L^2 (\Gamma \backslash G)}^2}{T \theta^2}  + \vol (B_G (e, 2T)^2) \vol (\Gamma \backslash G)_{<T} \|b_r\|_\infty^2.$$

\subsection{Conclusion of the proof} 
Adding back the index $n$ and integrating in $r$ we have
\begin{multline*}\left\| \frac{1}{T} \int_0^T \rho_{\Gamma_n \backslash G} (k_t )\, \mathbf{A}^{(n)} \,\rho_{\Gamma_n \backslash G} (k_t ) \, dt \right\|_{\rm HS}^2 \\
\lesssim_M \frac{\sup_{r\leq M} \|b_r^{(n)} \|_{L^2 (\Gamma_n \backslash G)}^2}{T \theta^2}  + \vol (B_G (e, 2T))^2 \vol (\Gamma_n \backslash G)_{<T} \sup_{r\leq M}  \|b_r^{(n)}\|_\infty^2.
\end{multline*}
Since $ \|b_r^{(n)} \|_{L^2 (\Gamma \backslash G)}^2 = O (  \|b_r^{(n)}\|_\infty^2 \mathrm{vol} (\Gamma_n \backslash G))$, the function $b_r^{(n)}$ is uniformly bounded and $N(\delta_n , \Gamma_n) \gtrsim  p(s_0) \delta_n \mathrm{vol} (\Gamma_n \backslash G)$, with $p(s_0) \neq 0$, choosing $T=\frac14 r_n$, with $r_n$ as in \eqref{limBS}, we get that
$$\frac{1}{N(\delta_n , \Gamma_n)} \left\| \frac{1}{T} \int_0^T \rho_{\Gamma_n \backslash G} (k_t )\, \mathbf{A}^{(n)} \,\rho_{\Gamma_n \backslash G} (k_t ) \, dt \right\|_{\rm HS}^2  \lesssim \frac{1}{r_n \delta_n} + \delta_n^{-1} \mathrm{vol} (B_G (e, r_n)) \alpha_n.$$
We have used here that $\mathrm{vol} (B_G (e, r_n/2))^2 \lesssim \mathrm{vol} (B_G (e, r_n))$. It now follows from \eqref{limBS} and $\delta_n \geq r_n^{-1}$ that the right hand side tends to $0$ when $n$ tends to infinity. This concludes the proof of Theorem \ref{T2}.

\section{A random viewpoint}\label{s:randomization}

In this section we prove Theorem \ref{I:T5}. This is done by considering two processes associated to a fixed spectral window $I$ and a compact quotient $\Gamma\backslash X$. Here  $\Gamma$ is a cocompact lattice in $G$ and $I \subset \ag^*$ is a compact subset.

\subsection{Two random processes}

Consider an orthogonal basis $\phi_1,\ldots, \phi_k$ of $\Gamma$-invariant eigenfunctions of the Laplacian on $X$ with eigenvalues in $I$ and such that $\| \phi_j \|_2 = \sqrt{\vol(\Gamma \backslash X)}$. Linear combinations of the (deterministic) eigenfunctions $\phi_1 , \ldots , \phi_k$ yield a map 
\begin{equation} \label{Process2}
\R^k  \to C^\infty (X); \ (c_1 , \ldots , c_k ) \mapsto \frac1{|c|} \sum_{j=1}^k c_j \phi_j,
\end{equation}
where $|c| = \left(\sum_{j=1}^k c_j^2\right)^{1/2}$. Putting on each factor $\R$ of $\R^k$ the Gaussian measure with mean $0$ and variance $1/ k$, and pushing this measure forward we get a measure $\lambda_{\Gamma , I} \in \mathbb{M}^1 (C^\infty (X))$. It is well known that if $X = (X_1,\ldots,X_k)$ is a vector of independent Gaussian random variables of mean $0$, then $X/|X|$ follows a uniform probability distribution on the unit sphere, hence $\lambda_{\Gamma , I}$ is the distribution of a random variable in $C^\infty (X)$ consisting in choosing uniformly at random a function in the unit sphere of $$\mathrm{span} \{ \phi_j :  s_j \in I \}.$$
The measure $\lambda_{\Gamma , I}$ is a Gaussian measure but it is not $G$-invariant. However --- being a random linear combination of $\Gamma$-invariant functions --- the measure $\lambda_{\Gamma , I}$ is supported on $\Gamma$-invariant functions $C^\infty (X)^\Gamma$. 
We can therefore BS-sample it to get a $G$-invariant measure on $\mathbb{M}^1 (C^\infty (X))$. One can think of two different ways to BS-sample it:
\begin{itemize}
\item First, take a $\lambda_{\Gamma , I}$-random $\phi$ in $C^\infty (X)$, and look at it from a random
point. The result is a $\lambda_{\Gamma , I}$-random element in $\mathbb{M}^1 (C^\infty (X))$, i.e. a probability measure $\alpha_{\Gamma , I}$ on $\mathbb{M}^1 (C^\infty (X))$.
To define it properly, consider the `BS-sampling' map
$$\Phi : C^\infty (X)^\Gamma \to \mathbb{M}^1 (C^\infty (X)); \ \phi \mapsto \mu_\phi$$
that associates to any $\Gamma$-invariant function $\phi \in C^\infty (X)$, the $G$-invariant probability measure $\mu_{\phi}$ on $C^\infty (X)$. We set
\begin{equation} \label{MesA}
\alpha_{\Gamma , I} = \Phi_* (\lambda_{\Gamma , I} ).
\end{equation}

\item Alternatively, take a random element $g$ in $G/\Gamma$, with respect to the Haar measure $\mu$, and restrict the
$\lambda_{\Gamma , I}$-random function $\phi$ on balls $g^{-1} (B(e , R))$. The result is a
$\mu$-random element of $\mathbb{M}^1 (C^\infty (X))$, so it is also a probability measure $\beta_{\Gamma , I}$ on $\mathbb{M}^1 (C^\infty (X))$.
To define it properly, consider the map
$$\Psi : G/ \Gamma \to \mathbb{M}^1 (C^\infty (X)); \ g\Gamma \mapsto (L_g)_* \lambda_{\Gamma , I},$$
where $L_g : C^\infty (X) \to C^\infty (X)),$ $\phi \mapsto \phi (g^{-1} \cdot )$, $(g \in G)$ is the left-translation map. We set
\begin{equation} \label{MesB}
\beta_{\Gamma , I} = \Psi_* (\mu).
\end{equation}
\end{itemize}

\begin{unremark} 
In general these two processes are not the same. To see this on a simplified situation replace the measure $\lambda_{I , \Gamma}$ by the measure $\lambda$ that takes the
constant $1$ function with probability $1/2$ and the constant $-1$ function with probability $1/2$.

Then the measure $\alpha$, obtained by the first process, will be the Dirac measure on constant $1$ function with probability $1/2$ and the
Dirac on the constant $-1$ function with probability $1/2$. But the measure $\beta$, obtained by the second process, will be the Dirac measure
on $\lambda$. The first is not a Dirac measure, the second is, so even for this very simple example,
the two processes are not equal.
\end{unremark}

However, the expected values of $\alpha_{\Gamma , I}$ and $\beta_{\Gamma , I}$ are always equal. We denote by $\nu_{\Gamma , I}$ their common expected values; it is the element of $\mathbb{M}^1 (C^\infty (X))$ obtained by pushing forward the product of the Gaussian measure by the normalized Haar measure by the map
\begin{equation} \label{RIBH}
\R^k \times G/\Gamma \to C^\infty (X)
\end{equation}
which maps $((c_1 , \ldots , c_k ),g)$ to $\sum_{j=1}^k c_j \phi_j (g^{-1} \cdot )$. 

\subsection{The theorem} Let $\Gamma_n $ be a uniformly discrete sequence of lattices in $G$ that BS-converges toward the trivial group. 
Let $\{ \phi_1^{(n)} , \ldots , \phi_{k_n}^{(n)} \}$ be a normalized orthonormal basis of the subspace 
of $C^\infty (X)^{\Gamma_n }$ spanned by eigenfunctions with eigenvalues in some interval $I_n = [s_0 - \delta_n , s_0 +\delta_n]$ with $\delta_n$ satisfying the condition stated after \eqref{limBS}. Suppose furthermore that $\beta'$ in the definition of $\delta_n$ is sufficiently small. Then we have the following theorem on the asymptotics of the measures $\alpha_{\Gamma_n , I_n}$ and $\beta_{\Gamma_n , I_n}$ on $\mathbb{M}^1 (C^\infty (X))$.

\begin{theorem} \label{T1}
The measures $\alpha_{\Gamma_n , I_n}$ and $\beta_{\Gamma_n , I_n}$ both weakly converge toward the Dirac mass concentrated at the Gaussian random wave $\mu_{{\rm Gauss} , s_0} \in \mathbb{M}^1 (C^\infty (X))$.
\end{theorem}
In particular the sequence of their common expected values $\nu_{\Gamma_n , I_n}$ weakly converges toward the Gaussian random wave $\mu_{{\rm Gauss} , s_0}$.

To prove Theorem \ref{T1} we will show that, in the weak limit, the second process $\beta_{\Gamma_n , I_n}$ is a
Dirac measure. Namely, $\lambda_{\Gamma_n , I_n}$ looks the same from most points, locally.
So, when we take expected value of the second process then we just have to erase the
Dirac symbol. To deal with the first process we shall use the ergodicity of the Gaussian random wave $\mu_{{\rm Gauss} , s_0}$.

\subsection{Covariance kernel}
We first fix $\Gamma$ and $I$ and compute the covariance kernel of the Gaussian process $(X^{\Gamma , I}_f)_{f \in \mathcal{D} (X)}$ associated to the measure $\lambda_{\Gamma , I} \in \mathbb{M}^1 (C^\infty (X))$.

\begin{lem} \label{L:cov}
Let $f,g \in \mathcal{D} (X)$. The covariance kernel is given by
$$\mathbb{E} ( X^{\Gamma , I}_f X^{\Gamma , I}_g ) = \int_{X \times X} K_{I , \Gamma}  (z, w)  f (z) g(w) dz dw$$
where 
$$K_{I, \Gamma} (z,w) = \frac{1}{k} \sum_{i=1}^k \phi_i (z) \otimes \phi_i (w).$$
\end{lem}
\begin{proof}
By definition of $\lambda_{\Gamma , I}$ we have:
\begin{align*}
\int_{\mathcal{E} (X)} T(f) T(g) d \lambda_{\Gamma , I} (T) & = \int_{\R^k} \left( \sum_{i=1}^k c_i (\phi_i ,  f )_{L^2 (X)} \right) \left( \sum_{i=1}^k c_i (\phi_i , g)_{L^2 (X)} \right)  dc  \\
\end{align*}
where $dc$ denotes the product of $k$ Gaussian measures on $\R$ with mean $0$ and variance $1/k$. In particular we have
$$ \int_{\R^k} c_i^2 \, dc = \frac1k$$
and $$  \int_{\R^k} c_i c_j \, dc = 0$$
if $i \neq j$. We thus have
\begin{align*}
\int_{\mathcal{E} (X)} T(f) T(g) d \lambda_{\Gamma , I} (T)
& = \frac{1}{k} \sum_{i=1}^k (\phi_i ,  f )_{L^2 (X)} (\phi_i , g)_{L^2 (X)} \\
& = \frac{1}{k} \sum_{i=1}^k \left( \int_{X} f(z) \phi_i (z) dz \right) \left( \int_{X} g(w) \phi_i (w) dw \right)  \\
& = \int_{X \times X} K_{I , \Gamma}  (z, w) f (z) g(w) dz dw.
\end{align*}
\end{proof}

\subsection{Asymptotics of $\beta_{\Gamma_n , I_n}$}
Now let $\Gamma_n $ be a uniformly discrete sequence of lattices in $G$ that BS-converges toward the trivial group and let $\{ \phi_1^{(n)} , \ldots , \phi_{k_n}^{(n)} \}$ be a normalized orthonormal basis of the subspace  of $C^\infty (X)^{\Gamma_n}$ spanned by eigenfunctions with eigenvalues in some interval $I_n = [s_0 - \delta_n , s_0 +\delta_n]$ with $\delta_n$ as in Conjecture \ref{C1}. Suppose furthermore that $\beta'$ in the definition of $\delta_n$ is sufficiently small. We first prove:

\begin{prop} \label{L:T1a}
The measures $\beta_{\Gamma_n , I_n}$ weakly converge toward the Dirac mass concentrated at the Gaussian random wave $\mu_{{\rm Gauss} , s_0} \in \mathbb{M}^1 (C^\infty (X))$.
\end{prop}
\begin{proof} The space $\mathbb{M}^1 (C^\infty (X))$ equipped with the weak* topology is a metrizable space. Indeed the topology is induced by the L\'evy-Prokhorov metric (see for example \cite[p. 72]{Bil99}). Note that here $C^\infty (X)$ is equipped with the usual compact convergence topology which makes it a Polish space. We prove the following lemma below.

\begin{lem} \label{L:T1abis}
There exists a sequence $(\delta_n )$ of positive real numbers that converges to $0$ and satisfies the following property: 
given a distance $d$ on $\mathbb{M}^1 (C^\infty (X))$ that induces the weak* topology, the expected distance 
$$\int_{G/\Gamma_n} d((L_g)_* \lambda_{\Gamma_n , I_n} , \mu_{{\rm Gauss} , s_0}) d \dot g,$$  
between a random $G$-translate of $\lambda_{\Gamma_n , I_n}$ and $\mu_{{\rm Gauss}, s_0}$, tends to $0$ as $n$ tends to infinity. Here we indicate by a dot in $d\dot g$ that the Haar measure is normalized.
\end{lem}

Let us now prove that Lemma \ref{L:T1abis} implies Proposition \ref{L:T1a}. The weak convergence of $\beta_{\Gamma_n , I_n}$ to $\delta_{\mu_{{\rm Gauss} , s_0}}$ is equivalent to having for any open subset $U$ of $\mathbb{M}^1 (C^\infty (X))$
$$ \liminf_n \beta_{\Gamma_n , I_n} (U) \geq \delta_{\mu_{{\rm Gauss} , s_0}} (U).$$
Let $U$ be an open subset of $\mathbb{M}^1 (C^\infty (X))$. If $U$ does not contain $\mu_{{\rm Gauss} , s_0}$ we obviously have
$$\liminf_n \beta_{\Gamma_n , I_n} (U) \geq 0 = \delta_{\mu_{{\rm Gauss} , s_0}} (U).$$
Suppose now that $U$ contains $\mu_{{\rm Gauss} , s_0}$. Then $U$ contains a small open ball $B ( \mu_{{\rm Gauss} , s_0} , \eta )$. Let $\varepsilon$ be a positive real number. 
Since 
$$\int_{G/\Gamma_n} d((L_g)_* \lambda_{\Gamma_n , I_i} , \mu_{{\rm Gauss} , s_0}) d \dot g \to 0$$  
as $n \to +\infty$, there exists some positive integer $N_0$ such that for every $n \geq N_0$, we have:
$$\mathrm{vol} \left\{ g \in G/\Gamma_n \; | \; d((L_g)_* \lambda_{\Gamma_n , I_i} , \mu_{{\rm Gauss} , s_0}) \geq \eta \right\} \leq \varepsilon.$$
It follows that 
\begin{equation*}
\begin{split}
\beta_{\Gamma_n , I_n} (U) & = \mathrm{vol} \left\{ g \in G/\Gamma_n \; | \; (L_g)_* \lambda_{\Gamma_n , I_i}  \in U \right\} \\
& \geq \mathrm{vol} \left\{ g \in G/\Gamma_n \; | \; d((L_g)_* \lambda_{\Gamma_n , I_i} , \mu_{{\rm Gauss} , s_0}) < \eta \right\} \\
& \geq 1 -\varepsilon.
\end{split}
\end{equation*}
In other words, $\liminf_n \beta_{\Gamma_n , I_n} (U) \geq 1 = \delta_{\mu_{{\rm Gauss} , s_0}} (U)$ and the proposition follows.
\end{proof}

\subsection{Proof of Lemma \ref{L:T1abis}}
The measure $\mu_{{\rm Gauss}, s_0}$ and any given translate of $\lambda_{\Gamma_n , I_n}$ are Gaussian measures, it follows in particular from the Bochner-Minlos theorem that they are determined by their characteristic functional, or equivalently by their covariance kernel. 
Recall that the characteristic functional of a Gaussian measure $\mu$ on $C^\infty(X)$ is given by
$$ \hat \mu(f) = \exp \left( -\frac12 \iint_{X \times X} f(x) K(x,y) f(y) \, dx \, dy \right), $$
where $K$ is the covariance kernel of $\mu$ and $f\in C^\infty(X)$.

Moreover, Paul L\'evy continuity theorem for generalized random fields (due to Fernique \cite{Fernique}) implies that weak* convergence on measures corresponds to simple convergence of characteristic functionals or equivalently the topology of uniform convergence of the kernels on compact subset. 
As the kernel of $(L_g)_* \lambda_{\Gamma_n , I_n}$ is given by $K_{I_n,\Gamma_n}(g^{-1}x, g^{-1}y)$ (see Lemma \ref{L:cov}) and the kernel of $\mu_{{\rm Gauss}, s_0}$ is the spherical function $\varphi_{s_0}(x^{-1} y)$ (see Lemma \ref{l:covkernel}), to prove Lemma \ref{L:T1abis} it suffices to show that for any compact set $C$ in $X$, the average 
\begin{equation} \label{average}
\int_{G/\Gamma_n } \left| K_{I_n,\Gamma_n} (h^{-1} x, h^{-1} y) - \varphi_{s_0}(x^{-1} y) \right| \, d\dot h
\end{equation}
of the difference between the covariance kernels of the $G$-translates of $\lambda_{\Gamma_n , I_n}$ and the covariance kernel of the Gaussian wave associated with $s_0$, converges to $0$ with $n$ uniformly for $x$ and $y$ in $C$. 

Given a positive $\delta$ such that $I_\delta = [s_0 - \delta , s_0 + \delta ]$ is contained in the tempered spectrum, we first remark that for any test function $\widehat{F}$ on $\mathfrak{a}^*$ --- with $F : K\backslash G / K \to \R$ compactly supported --- we have
\begin{equation} \label{E:T1a1}
\int_{G/\Gamma_n} \left| \sum_{j} \widehat{F} (s_j ) \phi_j^{(n)} (h^{-1} x) \phi_j^{(n)} (h^{-1} y) - \int \widehat{F} (s) \varphi_s (x^{-1} y) \, d\mu_{\text{Planch}}(s) \right| \, d\dot{h} \to 0 ,
\end{equation}
when $n\to+\infty$. Here $(\phi_j^{(n)})_j$ is an orthonormal basis of $L^2 (\Gamma_n \backslash X)$ that consist of eigenfunctions whose eigenvalue corresponds to $s_j \in \mathfrak{a}^*$. 

The proof of \eqref{E:T1a1} goes as follows: first rewrite the kernel appearing on the left hand side of \eqref{E:T1a1} as 
\begin{equation} \label{E:2side}
\sum_{j}  \widehat{F} (s_j ) \phi_j^{(n)} (x) \phi_j^{(n)} (y) = \sum_{\gamma \in \Gamma_n} F (x^{-1} \gamma y).
\end{equation}
Now we separate the right-hand side of \eqref{E:2side} into two terms:
$$\sum_{\gamma \in \Gamma_n} F (x^{-1} h \gamma h^{-1} y)  =  F (x^{-1} y)  
 +  \sum_{\gamma \in \Gamma_n, \gamma\neq e} F (x^{-1} h\gamma h^{-1} y).$$
We then integrate both sides of \eqref{E:2side} over $G/\Gamma_n$. The right-hand side yields
\begin{multline*}
 \int_{G/\Gamma_n} \sum_{\gamma \in \Gamma_n} F (x^{-1} h \gamma h^{-1} y) \, d\dot{h} = \int_{G/\Gamma_n} F (x^{-1} y) \, d\dot{h} \\
+  \int_{G/\Gamma_n} \sum_{\gamma \in \Gamma_n, \gamma\neq e} F (x^{-1} h\gamma h^{-1} y) \, d\dot{h}.
\end{multline*}
The first term is just 
$$F (x^{-1} y) = \int \widehat{F} (s) \varphi_s (x^{-1} y) \, d\mu_{\text{Planch}}(s). $$
To conclude, note that in the second integral, we can restrict the integration to the set 
$$\{ h \in G/\Gamma_n : \text{InjRad}(h^{-1}y) \leq d(x,y) + r \},$$
where $\mathrm{supp} (F) \subset B_G (e , r)$. It then follows from BS convergence that the average, over $G/\Gamma_n$, of the module this second integral tends to $0$ as $n$ tends to infinity, see e.g. \cite{7samuraiCRAS} for more details. This concludes the proof of \eqref{E:T1a1}.

Sauvageot density principle \cite[Thm. 7.3(b)]{Sauvageot} (see also \cite[Prop. 6.4]{7samurai}) then implies that \eqref{E:T1a1} holds with $\widehat{F}$ replaced by the characteristic function of $I_\delta$. We then conclude the proof as in Lemma \ref{L:deltan} by defining the sequence $(\delta_n)$ so that $\delta_n$ is the infimum over all positive $\delta$ such that 
\begin{equation*}
\int_{G/\Gamma_n} \left| K_{I_\delta , \Gamma_n}(h^{-1}x,h^{-1}y) - \int_{I_\delta} \varphi_s (x^{-1} y) \, d\mu_{\text{Planch}}(s) \right| \, d\dot{h} \leq \delta
\end{equation*}
uniformly for $x$ and $y$ in $B_G ( e, \delta^{-1})$.
\qed

\subsection{Conclusion of the proof of Theorem \ref{T1}}
It remains to prove that the measures $\alpha_{\Gamma_n , I_n}$ weakly converge toward the Dirac mass concentrated at the Gaussian random wave $\mu_{{\rm Gauss} , s_0} \in \mathbb{M}^1 (C^\infty (X))$. We first want to extract a converging subsequence from $\alpha_{\Gamma_n , I_n}$. This is possible because measures of the form $\alpha_{\Gamma_n , I_n}$ are defined on a compact set of measures, and so they form a compact space.
\begin{lem}
The set of BS-samplings of superpositions of eigenfunctions with eigenvalues in a bounded interval $I$, such as in \eqref{Process2} is sequentially compact for the weak topology.
\end{lem}

\begin{proof}
Let $\mu_{\psi_n}$ be a sequence of BS-samplings of superpositions of $\Gamma_n$-invariant eigenfunctions, where $\Gamma_n$ is any sequence of co-compact subgroups. The measures $\mu_{\psi_n}$ are supported on $C^\infty(X)$, equipped with the usual compact convergence topology. In particular $C^\infty(X)$ is a Polish space and by Prokhorov theorem, $\mu_{\psi_n}$ is compact iff it is uniformly tight. We will therefore show that $\mu_{\psi_n}$ is uniformly tight.

Let $\epsilon > 0$. We need to find a compact set $K_\epsilon \subset C^\infty(X)$ such that
$$ \forall n \in \N \quad \mu_{\psi_n}( K_\epsilon) \geq 1- \epsilon. $$
Recall that $C^\infty(X)$ has the Heine-Borel property and the topology is defined by the family of semi-norms
$$ p_N(f) = \max \{ |D^\alpha f(x)| : x \in K_N , |\alpha| \leq N \}$$
where $K_N$ is an exhaustion by compact sets of $X$ and $\alpha$ is a multi-index (See for example \cite[Section 1.46]{Rud91}).
Therefore for any arbitrary sequence of non-negative real numbers $(C_N)_{N\in\N}$ the set
$$ K(C_N) = \{ f\in C^\infty(X) :  \forall N \in \N, \: p_N(f) \leq C_N \} $$
is compact.
We have
\begin{align*}
\mu_{\psi_n}(C^\infty(X) - K(C_N)) &= |\{ g \in \Gamma_n \backslash G \; : \; \psi_n(g^{-1} \cdot) \notin K(C_N) \} | \\
& = | \{ g\in \Gamma_n \backslash G \; : \; \exists N, \: p_N(\psi_n(g^{-1} \cdot)) > C_N \} |\\
&\leq \sum_N | \{ g\in \Gamma_n \backslash G \;  :  \; p_N(\psi_n(g^{-1} \cdot)) > C_N \}| \\
&\leq \sum_N \frac1{C_N} \int_{\Gamma_n \backslash G} p_N(\psi_n(g^{-1} \cdot)) \, d\dot g
\end{align*}
where $d\dot g$ is the normalized Haar measure on $\Gamma_n \backslash G$ and the last line is obtained from Markov's inequality. Now by a Sobolev embedding theorem \cite[Theorem 3.4]{Hebey}, for each $N$ there exists $M$ such that
$$p_N(\psi_n(g^{-1} \cdot)) \lesssim_N \sum_{|\alpha| \leq M} \int_{K_N} | D^\alpha \psi_n(g^{-1}x)|^2 dx $$

Because $\psi_n$ is a normalized superposition of eigenfunctions with eigenvalues bounded by a uniform value $\lambda > 0$ (upper bound of the spectral interval $I$) we have
$$|D^\alpha \psi_n| \leq \lambda^{ \frac12 |\alpha|} | \psi_n|$$
from which we deduce that
$$p_N(\psi_n(g^{-1} \cdot)) \lesssim_N M^d \lambda^{\frac{M}2} \int_{K_N} | \psi_n(g^{-1}x)|^2 dx. $$
Because of the normalization of the eigenfunctions we also have for any $x \in X$
$$\int_{\Gamma_n \backslash G} |\psi_n(g^{-1} x)|^2 \, d \dot g = 1,$$
so by Fubini's theorem we deduce
$$\int_{\Gamma_n \backslash G} p_N(\psi_n(g^{-1} \cdot)) \, d\dot g \leq C(I,N), $$
where $C(I,N)$ is a constant depending only on the interval $I$ and $N$. We can now take $C_N = \frac{2^{N+1} C(I,N)}{\epsilon}$ so that for all $n\in\N$
$$\mu_{\psi_n}(C^\infty(X) - K(C_N)) \leq \epsilon, $$
which shows the uniform tightness of $\{\mu_{\psi_n}\}$.
\end{proof}

Now let $\alpha$ be any weak limit of the sequence $\alpha_{\Gamma_n , I_n}$. The measure $\alpha$ is supported on the subset of $\mathbb{M}^1 (\mathcal{F} (X))$ that consists of $G$-invariant measures. 

The expected value of $\alpha_{\Gamma_n , I_n}$ is equal to the expected value of $\beta_{\Gamma_n , I_n}$. Since the latter is supported on a bounded set, the weak convergence of the sequence of measures $\beta_{\Gamma_n , I_n}$ (Lemma \ref{L:T1abis}) implies the convergence of the expected values. It follows that the sequence of expected values $\mathbb{E} (\alpha_{\Gamma_n , I_n})$ weakly converges toward the Gaussian random wave $\mu_{{\rm Gauss} , s_0} \in \mathbb{M}^1 (C^\infty (X))$. We conclude that 
$$\mathbb{E} (\alpha ) = \mu_{{\rm Gauss} , s_0}.$$
In particular the measure $\alpha$ is supported on smooth functions, and, since $\mu_{{\rm Gauss} , s_0}$ is ergodic, we conclude that $\alpha$ is equal to the Dirac mass on $\mu_{{\rm Gauss} , s_0}$. This proves Theorem~\ref{T1}.  

\section{Some problems}\label{s:problems}

In this section we list some open problems motivated by the Benjamini-Schramm viewpoint on quantum chaos. 
We do not repeat Conjectures \ref{BerryConj} and \ref{BC2} from the Introduction here. We try to choose the wildest possible interpretations to stimulate finding (counter)examples.

\bigskip

Our formulation of Berry's conjecture allows us to decompose the problem to
smaller ones and point out some baby cases to be attacked.

Let $M$ be a $d$ dimensional compact manifold. We call an invariant random function $F$ a \emph{Wigner wave for} $M$ if there exists a sequence of
($L^2$-normalized) eigenvectors $(\phi _{n},\lambda _{n})$ with $\lambda _{n}\rightarrow \infty 
$ such that $(M,\phi _{n},\lambda _{n})$ Benjamini-Schramm converges to $F$.
Berry's conjecture says that when $M$ has negative curvature, the only
Wigner wave for $M$ is the Gaussian monochromatic wave.

Let us define the \emph{invariant sine wave} on $\mathbf{R}^{2}=\mathbf{C}$
by 
\[
\mathrm{IS}_{a,\varepsilon }(z)=\sin (\left\langle \varepsilon
,z\right\rangle + a) 
\]%
where $\varepsilon $ is uniform random in the unit circle and $a$ is uniform
random on $[0,2\pi ]$. That is, $\mathrm{IS}$ is a random translation of a
random rotation of $\sin $. Note that $\mathrm{IS}(z)$ is equal to the
Benjamini-Schramm sample (a random lift) of the $\sin $ function on the
standard torus.

The following is a baby case of Berry's conjecture. It may serve as a first
step to understand the role of negative curvature in the picture.

\begin{problem}
Let $M$ be a compact negatively curved surface. Show that $\mathrm{IS}$ is
not a Wigner wave for $M$.
\end{problem}

This roughly translates to saying that $M$ does not admit high energy
eigenfunctions that locally, at most points, look like the sine wave. 

\medskip

The following is a weak version of Conjecture \ref{BerryConj}. 

\begin{problem}
Let $M$ be a compact manifold. Is the monochromatic Gaussian eigenwave
a Wigner wave for $M$? 
\end{problem}

A first attempt would be to try and analyze the behavior of a random
eigenfunction in a shrinking window, opposed to a Gaussian random sum of
these eigenfunctions, which in the level aspect is settled in Theorem \ref{T1}. 

\medskip

Our next problem addresses the conservation of energy in Quantum Ergodicity.
A priori, it can happen that the distribution of values loses mass at
infinity. As we show in Proposition \ref{P15}, assuming that this will not happen implies that
any limit of the square measures is absolutely continuous with respect to
volume. Note that the other implication is not clear, as the high energy
places could very well equidistribute enough to admit volume as a weak limit.

\begin{problem}
Is there a Wigner wave for a negatively curved compact manifold with energy $<1$?
\end{problem}

\medskip

To conclude this section, we believe that the Benjamini-Schramm viewpoint should also be useful to study nodal
domains. This study has attracted a lot of research recently (see e.g. \cite{GW,NS,NaliniBourbaki,Sarnak}). Recall that if $M$ is a manifold, then the zero set of an
eigenfunction of $M$ cuts $M$ into pieces called nodal domains. A very general (and vague) problem is to 
analyze the shape and number of nodal domains for Benjamini-Schramm
convergent sequences of manifolds. 

We do not expect a straight continuity result here, that is, nodal
domains will not be entirely local. However \cite[Proposition 2]{Ingremeau} already shows that Benjamini-Schramm convergence can be used to give lower bounds on the number of nodal
domains of a family of eigenfunctions.

\end{document}